\documentclass{svjour3}                     
\usepackage{color}
\definecolor{mycolor}{RGB}{0, 0, 0}
\definecolor{mycolor2}{RGB}{0, 0, 0}
\definecolor{mycolor3}{RGB}{0, 0, 0}

\smartqed  
 \usepackage{paralist}
\usepackage{graphicx}
\usepackage{amsmath}
 \usepackage{mathptmx}      
\usepackage{amsfonts}
\usepackage{graphicx}
\usepackage{url}
\usepackage{float}
\usepackage{latexsym}
\usepackage{amssymb}
\usepackage{verbatim}
\usepackage{amsmath}
\usepackage{xcolor}
\usepackage{bbding}
\usepackage{hhline}
\usepackage[math]{cellspace}
\usepackage{latexsym}

\begin{document}
\title{A Non-commutative Cryptosystem Based on Quaternion Algebras
}


\author{
Khadijeh Bagheri       \and
  Mohammad-Reza Sadeghi \and
  Daniel Panario }


\institute{K. Bagheri  \at
              Faculty of Mathematics and Computer Science, \\
              Amirkabir University of Technology, Tehran, Iran\\
               \email{kbagheri@aut.ac.ir}        
             \and
           M.-R. Sadeghi \at
              Faculty of Mathematics and Computer Science, \\
              Amirkabir University of Technology, Tehran, Iran\\
              \email{msadeghi@aut.ac.ir}           
              \and
              D. Panario \at
             School of Mathematics and Statistics,\\
              Carleton University, Ottawa, Canada\\
              \email{daniel@math.carleton.ca}
}

\maketitle
\begin{abstract}
We  propose BQTRU, a  non-commutative NTRU-like cryptosystem  over  quaternion
algebras.  This  cryptosystem  uses  bivariate  polynomials as  the  underling  ring.
The multiplication  operation  in  our  cryptosystem  can  be  performed  with  high  speed  using quaternions  algebras  over finite  rings.  As a consequence,  the
key  generation  and  encryption  process  of  our  cryptosystem  is  faster  than  NTRU  in
comparable  parameters.  Typically  using  Strassen's  method,  the  key  generation  and
encryption  process  is  approximately  $16/7$  times  faster  than  NTRU  for  an  equivalent
parameter set. Moreover,  the BQTRU  lattice has a hybrid  structure that makes  inefficient
standard  lattice attacks  on  the  private  key. This entails a higher computational complexity for attackers providing the opportunity of having \textcolor{mycolor}{smaller key sizes}. Consequently, in this sense,  BQTRU  is  more  \textcolor{mycolor}{resistant  than  NTRU  against known attacks} at  an equivalent  parameter  set. Moreover, message  protection  is  feasible  through  larger
polynomials and  this allows us to  obtain  the  same  security  level  as other NTRU-like cryptosystems but using lower
dimensions.

\end{abstract}
\section{Introduction}
\noindent
The NTRU  public key cryptosystem was proposed by Hoffstein et al. in 1996 \cite{NTRU}. The first version of this public key cryptosystem was standardized and published within IEEE P1363.1 \cite{IEEE}. Its security is based on the hardness of finding the shortest vector problem (SVP) and the closest vector problem (CVP) in a particular type of lattices, the convolutional modular lattices \cite{May}.  Furthermore,  these two lattice problems are known to be NP-hard problems \cite{CVP}, \cite{SVP}. Thus,  NTRU  is a candidate for quantum-resistant public key cryptosystems.

The NTRU cryptosystem has been shown to be faster than RSA and elliptic curve cryptosystems \cite{speed}. Indeed, the linearity of lattice operations offers speed advantages over cryptosystems based on other known hard problems. Considering computational efficiency and low cost of implementation, NTRU is one of the fastest public key cryptosystems \cite{speed}.

NTRU is based on the polynomial ring $R=\mathbb{Z}[x]/\langle x^n -1 \rangle$. Coppersmith and Shamir \cite{Coppersmith} present a lattice attack against NTRU that finds the private key from its public key by solving the SVP on the NTRU lattice. They suggest that a non-commutative algebra  can be considered for the underlying algebra to avoid these attacks. Several non-commutative proposals have been introduced and, consequently, have been broken.  Gaborit et al. \cite{CTRU} introduce CTRU as an analogue to NTRU where the coefficients of polynomials are from $\mathbb{F}_2$ instead of $\mathbb{Z}$. Kouzmenko \cite{NTRUG} shows that CTRU  can be broken using a  polynomial time attack. Generalizations of NTRU over the Dedekind domains including $\mathbb{Z}[i]$, $\mathbb{Z}[\sqrt{-2}]$, $\mathbb{Z}[\zeta_3]$ and $\mathbb{Z}[\zeta_5] $, where $\zeta_3$ and $\zeta_5$ are $3$rd and $5$th roots of unity, are presented in \cite{ETRU} and \cite{NTRUZ} . A non-commutative version of NTRU over the non-commutative ring of $k\times k$ matrices of polynomials in the ring $R$ is also available in the literature \cite{MaTRU}. Another non-commutative version of NTRU, called QTRU,  is addressed  by Malekian et al. over quaternion algebras \cite{QTRU}, \cite{QTRU1}. The computational processes of  QTRU  have high complexity and at equivalent parameters set, QTRU is slower that NTRU.
In a different direction,  Caboara et al. \cite{GB-NTRU} increase the number of variables in the polynomial ring  $R$,  instead of operating in a different coefficient ring. \textcolor{mycolor2}{Moreover, they introduce an ideal of $R$ in their scheme so to complicate the attacking scenarios of NTRU-like cryptosystem. However, this scheme is inefficient due to the high complexity of the key generation, the encryption and the decryption algorithms. An improvement of this work is presented for bilinear polynomial rings in \cite{Orsini},  which is secure against the lattice based attack and has an efficient encryption algorithm but its key generation and decryption algorithm are still not efficient enough to be practical.}
\textcolor{mycolor2}{The proposed method in this paper, inspired by \cite{Orsini} and \cite{QTRU}, leads to an efficient NTRU-like cryptosystem that is secure against lattice based attacks in smaller dimensions.
}

There are other advantages and practical potentialities when using non-commutative structures \textcolor{mycolor}{like quaternions as the underlying ring of NTRU. On the one hand,} many applications in image processing are concerned with signals and images that have three or four dimensional samples. Quaternions are useful tools for modeling and analyzing such signals that arise very naturally in the physical world from the three dimensions of physical space to trichromatic nature of human color vision. More generally, the four dimensions of quaternions can be used to represent a most general set of geometric operations in three dimensions using homogenous coordinates. The reader is pointed to the recent book \cite{signal processing} that shows quaternions usages in signal and image processing. Given this number of applications of quaternions, looking at scenarios that require secure data transmission of a source that generates vector-valued samples, leads us to design a fast cryptosystem based on quaternions. \textcolor{mycolor}{On the other hand, under} special conditions,  quaternion algebras over a field $\mathbb{F}$ are isomorphic to $M_2(\mathbb{F})$, the ring of $2\times2$ matrices over $\mathbb{F}$. In this case, the multiplication of two elements of these quaternion algebras  is equivalent to a multiplication of their corresponding  $2\times2$ matrices over $\mathbb{F}$. \textcolor{mycolor}{Strassen shows} that  two  $2\times2$  matrices  can  be  multiplied using  only  seven multiplications and some additions \cite{Strassen}. Thus, any quaternion multiplication can be done with $7$ multiplications over  $\mathbb{F}$ instead of $16$. \textcolor{mycolor}{Therefore, if we use the special case of quaternion algebra in a cryptosystem, it leads to low complexity encryption and decryption algorithms due to the high speed operations in this algebra.} This also motivated us to present an extensions of NTRU cryptosystem over quaternions.

In this paper, we introduce \textcolor{mycolor}{an NTRU-like} cryptosystem called BQTRU. We change the underlying algebraic structure of NTRU to a quaternion algebra over the \textcolor{mycolor2}{bivariate polynomial} ring $ R'=\mathbb{Z}[x, y]/\langle x^n -1,y^n -1 \rangle$.
In  BQTRU,  four messages from four different sources are encrypted  at the same time; this is useful in multiuser applications. \textcolor{mycolor}{Moreover,} the processing in the BQTRU cryptosystem is fast thanks to several quaternion properties.

\textcolor{mycolor2}{Similar to \cite{Orsini} and \cite{GB-NTRU}, we use an ideal of $ R'$ when designing the BQTRU cryptosystem to enhance the security of our proposed NTRU-like cryptosystem against lattice attacks. Therefore,  the BQTRU  lattice also has a hybrid  structure that makes  inefficient standard lattice attacks on the  private  key.} Hence, the lattice based attacks on our private key have higher complexity than other NTRU-like cryptosystems. \textcolor{mycolor}{As a consequence of some special properties, BQTRU is more resistant} than other NTRU-like cryptosystems against lattice based attacks in smaller dimensions. Typically, we can obtain  the same security level of other NTRU-like cryptosystems but here in lower dimensions. \textcolor{mycolor}{This high complexity of lattice based attacks opens the possibility of smaller key sizes and more message protection with respect to NTRU.}

The rest of the paper is organized as follows. In Section \ref{preliminaries}, the original presentation of the NTRU cryptosystem and  required mathematical background about quaternion algebras are revisited.  The design of the proposed cryptosystem is presented in Section \ref{QTWO}.  In Section \ref{Successful Decryption}, we present the choice of  parameters
for BQTRU which makes the probability of  decryption failure to be close to negligible.   In Section \ref{Lattice attacks}, the security analysis of the proposed  cryptosystem  against  known attacks on  both  the  public  key   and  the  message are discussed. Using the parameters presented in the previous section, the  size  of the key space and message space of  BQTRU is compared with  NTRU and QTRU.
We also design a lattice based attack on BQTRU and  we show that BQTRU is resistant against this attack.
Conclusions are given  in Section \ref{concolusion}.
\section{Preliminaries}\label{preliminaries}
\subsection{The NTRU cryptosystem} \label{NTRU}
In this section, we briefly describe \textcolor{mycolor}{the original} NTRU public key cryptosystem; for more information, see \cite{book}, \cite{NTRU}. We follow the presentation in \cite{NTRU}. We start by fixing an integer $n > 1$ and two coprime moduli $p$ and $q $ with $q\gg p$. Commonly $n$ is chosen to be prime and $p$ is chosen to be $3$. Let $R=\mathbb{Z}[x]/\langle x^n -1 \rangle$ be the ring of truncated polynomials with convolutional multiplication denoted by the symbol $*$. In this ring, the addition of two polynomials is defined by the pairwise addition of the coefficients.
We consider the rings $R_{q}=\mathbb{Z}_q[x]/\langle x^n -1 \rangle$ and $R_{p}=\mathbb{Z}_p[x]/\langle x^n -1 \rangle$ and the subsets $L_{f}$, $L_{g}$, $L_{r}$ and $L_{m}$ of $R$. The set $L_{f}$ (respectively, $L_{g}$ and $L_{r}$) contains ternary polynomials, \textcolor{mycolor}{that is polynomials with coefficients $ +1$ and $-1$,} in $R$ with $2d_{f}$ (respectively, $2d_{g}$ and $2d_{r}$ ) monomials with coefficients $\pm1$, and the set $L_{m}$ contains polynomials in $R$  whose coefficients are between $-(p-1)/2$ and $(p-1)/2$. In NTRU the usual choice is $d_{f}=d_{g}=d_{r}=d$ approximately of size $n/3$.
\textcolor{mycolor}{ For correctness of the decryption algorithm in NTRU, we concentrate on polynomials of small norm, compared to the modulus $q$. We use the term ``small'' polynomial to indicate polynomials in $R$ having small norm  compared to the modulus $q$. Specially in the original NTRU, small polynomials are those ternary polynomials such that many coefficients are zero and the rest ($2d\approx 2n/3$) are in the set $\{1, -1\}$.} The usage of these type of polynomials is key for NTRU security.

Using public parameters $n, p$ and $q$, NTRU key generation selects two small random polynomials $f \in L_{f}$ and $g \in L_{g}$ such that $f$ is invertible in $R_{q}$ and $R_{p}$. The inverses of $f$ in $R_{q}$ and $R_{p}$ are denoted by $f^{-1}_{q}$ and $f^{-1}_{p}$, respectively.

 The public key is $h=f^{-1}_{q} \ast g \pmod q$.

 The private key is $(f, g)$.

 To encrypt a message $m$ from the set of messages $L_{m}$, we pick a random polynomial  $r \in L_{r}$ and compute
\begin{eqnarray*}
c=ph \ast r+m \;\;(\mbox{mod } q).
\end{eqnarray*}

To decrypt a ciphertext $c$, we use the private key $f$ and compute
\begin{eqnarray*}
e= f\ast c= f\ast(ph \ast r+m)
=  pf\ast (f^{-1}_{q}\ast g)\ast r+f\ast m = pg\ast r+f\ast m \;\;(\mbox{mod } q).
\end{eqnarray*}
We can ensure that all of the coefficients of $pg\ast r+f\ast m$ fall in the interval  $[-\frac{q-1}{2}, \frac{q-1}{2}]$ by a careful choice of parameters. Thus, $e \in R_{q}$ is equal to $ pg\ast r+f\ast m  \in R$ and the message $m$ is obtained from $e$ by reducing the coefficients of $f^{-1}_{p}\ast e $ modulo $p$. The public parameters $(n,p,q)$ of NTRU are chosen such that the probability of decryption failure is minimized.
If $q >(6d+ 1)p$, then the decryption process never fails \cite{NTRU}.

There are several known attacks on NTRU which are classified in two types of attacks:  combinatorial and lattice based attacks. The combinatorial attacks are easily avoided by a careful choice of parameters \textcolor{mycolor}{\cite{MITM}}; the lattice based attacks are the most significant attacks to date and lead to the more recent developments of NTRU  \cite{Coppersmith}. Coppersmith and Shamir introduced the NTRU lattice as a submodule $\mathcal{L}$ of $R\times R$ generated by $(q,0)$ and $(1,h)$. When we represent the elements of $R$ by their coefficient vectors in $\mathbb{Z}^n$, the submodule $\mathcal{L}$
can also be viewed as a $2n$-dimensional lattice $L_{\textrm{NTRU}}$ with row basis matrix
$$\left[
                       \begin{array}{cc}
                        qI  & 0 \\
                        \mathcal{H}  & I
                       \end{array}
                     \right],$$
 where $\mathcal{H}$ is a circulant $n\times n$ matrix defined by the coefficient vector of $h$ as first row, and each additional row is obtained from the one above it by cyclic shifting one position to the right.
 An attempt at finding the private key $(f,g)$ leads to solving SVP in $L_{\textrm{NTRU}}$ \cite{Coppersmith}, while an attempt at obtaining the message $m$ leads to finding CVP in $L_{\textrm{NTRU}}$ \cite{NTRU}.
 Lattice basis reduction algorithms like LLL \cite{LLL} and BKZ  \cite{BKZ} (or BKZ 2.0  \cite{BKZ2}) can find SVP in an  $n$ dimensional lattice $\Lambda\subseteq \mathbb{R}^n$ with length up to a factor of at most $2^{O(n)}$ the size of a shortest vector in $\Lambda$. Therefore, when $n$  gets  sufficiently large  the algorithms  fail  to  determine a suitable result for SVP. Thus the dimension of the lattice must be large enough in order to have a reasonable level of security in the cryptosystem.
\subsection{Quaternion algebra}~\label{Quaternion algebras}
In this section, we briefly review quaternion algebras which contain the underlying algebra of our proposed  cryptosystem as a special case. Hamilton discovered the real quaternion algebra $\mathbb{H}$ in 1843. It is known that the only real normed division algebras are  $\mathbb{R}$, $\mathbb{C}$, $\mathbb{H}$ and the Octonions $\mathbb{O}$ \cite{Conway}, where the normed division algebras are not necessarily commutative or associative. Among the above algebras, $\mathbb{H}$ is the only one that is associative and non-commutative. The real quaternions  can be considered as a vector space of dimension 4 over $\mathbb{R}$ with an ordered basis $\{1, i, j, k\}$ as follows
$$\mathbb{H}=\{ s_0 +s_1 i+s_2j+s_3 k ~|~ s_0, s_1, s_2, s_3 \in \mathbb{R}\}.$$
The addition of two quaternions and the scalar multiplication are defined by component-wise vector addition and scalar multiplication.
The multiplication of two quaternions denoted by $\circ$ is defined by the following rules:
\begin{eqnarray*}
i^2 =j^2=k^2=-1,\quad i\circ j=-j\circ i=k, \quad j\circ k=-k\circ j=i, \quad k\circ i=-i\circ k =j.
\end{eqnarray*}
We  define quaternion algebras over general fields whose characteristic are not $2$. Some of the basic properties of these algebras are derived using straightforward computational arguments.
\begin{definition}
Let $a$ and $b$ be non-zero elements of a field $\mathbb{F}$. The quaternion algebra $\mathbb{A}$ over $\mathbb{F}$ is a four
dimensional $\mathbb{F}$-space with basis $1$, $i$, $j$ and $k$, where the bilinear multiplication is defined by the conditions that $1$ is a unity element, and
\begin{eqnarray*}
i^2=a, \quad j^2=b,\quad i\circ j=-j\circ i=k.
\end{eqnarray*}
\end{definition}
Therefore
$$ k^2=-ab, \quad j\circ k=-k\circ j=-ib, \quad k\circ i=-i\circ k =-ja$$
and $\mathbb{A}$ is an associative algebra over $\mathbb{F}$.

For each quaternion $s= s_0 +s_1 i+s_2j+s_3 k$ in $\mathbb{A}$, its conjugate is given by $\bar{s}= s_0 -s_1 i-s_2j-s_3 k$, and its norm is defined by $N(s)=s\circ \bar{s}=\bar{s} \circ s=s_0^{2}-as_1^{2}-bs_2^{2}+abs_3^{2}$. A quaternion is a \emph{unit quaternion} if its norm is 1. The inverse of the quaternion $s$ is given by $s^{-1}=\bar{s}/N(s)=\bar{s} \circ (N(s))^{-1}$ provided that its norm is nonzero.

The  constructed algebra is denoted by $\Big(\frac{a, b}{\mathbb{F}}\Big)$ as a (generalized) quaternion algebra over $\mathbb{F}$. Hamilton's quaternions occur as the special case $\mathbb{H}=\Big(\frac{-1, -1}{\mathbb{R}}\Big)$.  It is well known that  $\Big(\frac{a, b}{\mathbb{F}}\Big)$ is a four-dimensional associative algebra with central field  $\mathbb{F}$.

A quaternion algebra over $\mathbb{F}$ is either a division algebra or else is isomorphic to $M_2(\mathbb{F})$. In \cite{quaternion isomorphism}, it is shown that for any field  $\mathbb{F}$,
$$M_2(\mathbb{F})\cong\Big(\frac{1, 1}{\mathbb{F}}\Big),$$
where the isomorphism is induced by
$$i= \left[ \begin{array}{cc}
                        1  & 0 \\
                       0  & -1
                       \end{array}\right],\qquad
 j= \left[ \begin{array}{cc}
                        0  & 1 \\
                       1  & 0
                       \end{array}\right].$$
The  \emph{Lipschitz} quaternions or \emph{integral quaternions}  $\mathbb{L}=\Big(\frac{a, b}{\mathbb{Z}}\Big)$
form a subring of the real quaternions $\mathbb{H}$. The algebra $\mathbb{L}_{q}=\Big(\frac{a, b}{\mathbb{F}_q}\Big)$ is defined over the finite field $\mathbb{F}_q$, where $q$ is  a prime number.

In general, the quaternion algebra $\mathbb{A}$ can be defined over any commutative ring $R$ with identity.
 In this paper, we define the quaternion algebra over  the polynomial ring  $R'=\mathbb{Z}[x, y]/\langle x^{n}-1, y^{n}-1 \rangle $, where we can perform arithmetic operations efficiently, low complexity and low memory requirements.

When we work with fields like the complex numbers, we have efficient algorithms for solving the equations with coefficients belonging to this field. In the quaternions, some problems like root finding algorithms, even solving a quadratic equation with coefficients belonging to $\mathbb{H}$, become nontrivial. Kalantari in  \cite{Kalantari}  suggests a new and simple approach for computing the solutions of a quadratic polynomial in $\mathbb{H}[x]$. We remark that even quadratic equations where the variable and the coefficients may not commute are fairly complicated to solve.

Factoring  a polynomial in $R[x]$ where $R$ is noncommutative does not preserve evaluation, that is, given $P, E, H \in R[x]$, the evaluation of the polynomials at $r \in R$ is not a ring homomorphism. That is,
\begin{center}
$P(x)=E(x)H(x)$ does not imply $P(r)=E(r)H(r)$.
\end{center}
\begin{example}
Consider $P(x)=(x-i)(x-j) \in \mathbb{H}[x]$. We might expect $i$ and $j$ to be roots of $P$. However, $P(x)= x^2-(i+j)x+k$. and hence
\begin{eqnarray*}
P(i)&=&i^2 -(i+j)\circ i+k= -1-i^2-i\circ j+k=2k\neq 0,\\
P(j)&=&j^2 -(i+j)\circ j+k= -1-i\circ j -j^2+k=0.
\end{eqnarray*}
Therefore, $j$ is a (right) root of $P$, but $i$ is not.
\end{example}
\begin{definition}[\cite{Kalantari}]
Two quaternions $s$ and $s'$ are congruent or equivalent, denoted by $s\sim s'$, if for some
quaternion $w \neq 0$ we have, $s' = w\circ s\circ w^{-1}$. The congruence class of $s = s_0 +s_1 i+s_2j+s_3 k$, denoted by $[s]$, is the set
\begin{eqnarray*}
[s]=\left\{s'\in \mathbb{H} ~|~ s'\sim s\right\}.
\end{eqnarray*}
\end{definition}
It is proved in \cite{Kalantari} that $[s]=\left\{ s_0+x_1 i+x_2 j+x_3 k ~|~  x_1^2+x_2^2+x_3^2=s_1^2+s_2^2+s_3^2 \right \}$.
 \begin{definition}[\cite{Kalantari}]
 Let $R$ be a noncommutative ring. An element $r\in R$ is a right root of a nonzero polynomial $P \in R[x]$ if and only if
$P(x) = E(x)(x-r)$ for some polynomial $E \in R[x].$
 \end{definition}
 \begin{theorem}[\cite{Kalantari}]
 Let $D$ be a division ring, and let $P(x)=E(x)H(x)$ with $P, E, H \in D[x]$. An element $d \in D$ is a root of $P$ if and only if  either $d$ is a root of $H$, or a conjugate of $d$ is a root of $E$.
 \end{theorem}
The fact that the evaluation map over the quaternions does not define a ring homomorphism from $\mathbb{H}[x]$ to $\mathbb{H}$, is a source of difficulty in polynomial root-finding over the quaternions. On the other  hand, in contrast with real or complex polynomials, a quaternion polynomial can have infinitely many roots or it could have fewer roots than its degree without multiplicity. This arises from the following theorem.
 \begin{theorem}[\cite{Gordon}]\label{infinitely root}
 Let $D$ be a division ring. A polynomial in $D[x]$ of degree $m$ \textcolor{mycolor}{can have infinitely many roots}, however these roots come from at most $m$ distinct conjugacy classes of $D$.
 \end{theorem}
A method has been presented for approximating the zeros of a real quaternion polynomial in \cite{Kalantari}. Moreover,  a simple method for constructing a quaternion polynomial with prescribed zeros and  a decomposition theorem for a quaternion polynomial in $\mathbb{H}[x]$  is also described in \cite{Kalantari}.

The proposed cryptosystem in this paper requires the study of quaternion algebras over commutative ring $R'=\mathbb{Z}[x, y]/\langle x^{n}-1, y^{n}-1 \rangle $. Factorization, decomposition  and classification of zeros of a quaternion polynomial in this ring have not been addressed in the literature  as far we know.

One of the important  structures in our cryptosystem is the ideal $J$ of  $\mathbb{A}$, quaternion algebra over $R'$, which we introduce in the next section.
The key generation process involves finding the inverse of a polynomial $F=f_0+ f_1 i+f_2 j+f_3 k \in \mathbb{A}$ modulo the ideal $J$ which requires an efficient division algorithm in general. Due to the lack of such algorithm in the literature, we consider a special form for ideal $J$ that makes the existence of a quaternion division algorithm inessential for us. \textcolor{mycolor}{We consider $J=Q+Q i+Q j+Q k$, where the ideal  $Q$ is a central ideal of  $\mathbb{A}$, that is $Q\subseteq Z(\mathbb{A})$,} where $Z(\mathbb{A})$ is the center of  $\mathbb{A}$, that is, the set of all elements in $\mathbb{A}$ that commute with every element. The construction method of $Q$ and $J$ is also presented in the next section. It is related to the choice of a set containing $n$-th roots of unity, roots of the polynomial $x^{n}-1$ in $\mathbb{A}$. Our construction method for ideal $J$ considers only central roots of $x^{n}-1$. \textcolor{mycolor}{If this restriction could be removed and non-central roots in the construction of $J$ could be used, then Theorem \ref{infinitely root} would provide enhanced security to our system by improving  the resistance to a brute force attack on the private ideal $J$.}
\section{BQTRU: a cryptosystem based on  bivariate polynomials and quaternions}~\label{QTWO}
In this section, we introduce our cryptosystem by using the quaternion algebras as its underlying algebraic structure.
Operations in this cryptosystem are performed in the ring of bivariate polynomials $ R'=\mathbb{Z}[x, y]/\langle x^{n_1}-1, y^{n_2}-1 \rangle$. \textcolor{mycolor}{It is possible to take two different exponents $n_1$ and $n_2$, however for simplicity we consider $n_1=n_2=n$ in the rest of the paper.}
We consider an order on monomials, like lexicographical order,  embedding every polynomial $f \in R'$ in $\mathbb{Z}^{n^2}$. This embedding maps a polynomial $f \in R'$  into a vector $\overrightarrow{f}=(f_{0}, f_{1},\ldots ,f_{n^2-1}) \in \mathbb{Z}^{n^2}$.

We choose the public parameters $(n , p, q)$ similar to \textcolor{mycolor}{ the original NTRU: $n$ is a prime number,} $p$ and $q$ are two different prime numbers such that $\gcd(p, q)=\gcd(n, q)=1$ where $q \gg p$ and  $L_{f}, L_{g},  L_{\phi}$ and $L_{m}$ are subsets of $R'$ \textcolor{mycolor}{that are defined similar to their corresponding subsets in NTRU. In the proposed cryptosystem, the integers $d_f$, $d_g$, $d_\phi$ and $d_m$  associated to these subsets are considered such that $d_f=d_g=d_\phi=d_m=d\approx n^2 /7$.} An additional condition is the choice of $n$ and $q$ such that $n |(q-1)$.

We choose an ideal $Q_{q}$ of $ R'_{q}=\mathbb{Z}_q[x, y]/\langle x^n-1, y^n-1\rangle$ and use this ideal for key generation and decryption in our cryptosystem. In the next section, we show that the use of this ideal leads to applying a different lattice for decryption.
\textcolor{mycolor}{Since $\langle x^n-1, y^n-1\rangle \subset Q_{q}$,} therefore $Q_{q}$ corresponds to the intersection of ideals of the form $\langle x-a, y-b\rangle$ where $(a, b)$ belongs to the set $E=\{(a, b)\in \mathbb{Z}_q\times \mathbb{Z}_q ~|~a^n=1,  b^n=1 \}$. \textcolor{mycolor}{ Let $T \subset E$ be the set of points for which there is at least one polynomial in $Q_{q}$ with no roots from $T$.}
Then, it is clear that the ideal $Q_{q}$ is equal to the set of polynomials that vanish on $E\backslash T$.

\textcolor{mycolor}{The condition $n |(q-1)$ yields the factorization of  $x^n -1$ into linear factors in $\mathbb{Z}_q$. Then, $E$ is a set of elements of $\mathbb{Z}_q \times \mathbb{Z}_q$ with cardinality  $n^2$.}

\textcolor{mycolor}{Obviously, $ R'_{q}$ is a vector space of dimension $n^{2}$ over $\mathbb{Z}_{q}$ where we can consider the Lagrange basis with Lagrange interpolants defined for each point $(a, b)$ of $E$ by
 \begin{eqnarray}\label{landa}
\lambda_{a, b}(x,y)=\frac{ab(x^{n}-1)(y^{n}-1)}{n^{2}(x-a)(y-b)}.
\end{eqnarray}
Any polynomial $f(x, y)$ of $R'_{q}$ can be identified by its values over $E$ as \cite{interpolation}
\begin{equation}\label{linear combination}
f(x, y)=\sum_{(a,b) \in E} {f(a, b) \ast \lambda_{a, b}(x, y)}.
\end{equation}
For each $(a, b) \in E$, the Lagrange interpolant $\lambda_{a, b}(x,y)$  is a polynomial that vanishes on $E$ except at $(a, b)$ where its value is equal to 1. Moreover, every $\lambda_{a, b}(x,y)$ have the property that for each $\alpha \in R'_{q} $
\begin{equation}\label{lambda}
\alpha(x, y)\ast \lambda_{a, b}(x,y)=\alpha(a, b) \ast \lambda_{a, b}(x,y).
\end{equation}
In this way,  when we see  $ R'_{q}$ as  a vector space over $\mathbb{Z}_{q}$, then $Q_{q}$ can be consider as a vector subspace of $ R'_{q}$. Since any polynomial $\rho \in Q_{q}$ vanishes outside of $T$, by considering (\ref{linear combination}), we have $\rho(x, y)=\sum_{(a,b) \in T}\rho(a, b) \ast \lambda_{a, b}(x,y)$.   Without loss of generality, we assume that the first $|T|$ elements of the set $E$ are exactly the points of $T$.
Hence,  each  $\rho \in Q_{q}$ is written as a linear combination of  $\{\lambda_{a_i, b_i}(x, y)\}_{i=1}^{|T|}$ and this set can be a basis of $Q_{q}$ as a vector subspace of $ R'_{q}$}.

On the other hand, $Q_{q}$,  as an ideal, can be generated by
\begin{equation} \label{sigma}
\sigma(x,y)=\sum\nolimits_{i=1}^{|T|} q_i\lambda_{a_i, b_i}(x,y),
\end{equation}
where
\textcolor{mycolor}{ for $i=1, \ldots, |T| $, $q_i $ is a nonzero element of ${\mathbb{Z}_q}$  chosen at random},
or given by any other polynomial which does not vanish on $T$.
We observe that $\sigma(x,y)\ast \lambda_{a_i, b_i}(x, y)(q_i)^{-1}=\lambda_{a_i, b_i}(x, y)$ for $i=0,\ldots, |T|$, therefore the ideal generated by $\sigma(x,y)$ contains our basis for the vector space  $Q_{q}$. Consequently, we can consider $Q_{q}=\langle\sigma(x, y)\rangle$. In this way, the ideal $Q=\langle q, \sigma(x, y)\rangle$ of $R'$ has \textcolor{mycolor}{ $\{q\}\cup \{\lambda_{a_i, b_i}(x, y)\}_{i=1}^{|T|}$} as its generators, where $T\subseteq E$.

In our cryptosystem, we consider the quaternion algebras $\mathbb{A}$, $\mathbb{A}_p$ and $\mathbb{A}_q$ over the rings $R'=\mathbb{Z}[x, y]/\langle x^n-1, y^n-1\rangle$, $R'_p=\mathbb{Z}_p[x, y]/\langle x^n-1, y^n-1\rangle$ and $R'_q=\mathbb{Z}_q [x, y]/\langle x^n-1, y^n-1\rangle$, respectively:
\begin{align*}
\mathbb{A} &= \Big(\frac{1, 1}{R'}\Big)=\big\{ f_0 +f_1 i+f_2j+f_3 k ~\big|~ f_0, f_1, f_2, f_3 \in R';\;\; i^2 =j^2=1, \;\; i\circ j=-j\circ i=k\big\},\\
\mathbb{A}_p &= \Big(\frac{1, 1}{R'_p}\Big)=\big\{ f_0+ f_1 i+f_2 j+f_3 k  ~\big|~  f_0,f_1,f_2,f_3 \in R'_p;\;\; i^2 =j^2=1, \;\; i\circ j=-j\circ i=k\big\},\\
\mathbb{A}_q &= \Big(\frac{1, 1}{R'_q}\Big)=\big\{ f_0+ f_1 i+f_2 j+f_3 k  ~\big|~  f_0,f_1,f_2,f_3 \in R'_q;\;\; i^2 =j^2=1, \;\; i\circ j=-j\circ i=k \big\},
\end{align*}
which are isomorphic to $M_2(R')$, $M_2(R'_p)$ and $M_2(R'_q)$, respectively. These quaternion algebras possess the properties seen in Section \ref{preliminaries}, except that some nonzero elements might have zero norm and then such elements do not have a multiplicative inverse \cite{Conway}.

The multiplication of two quaternions $F$ and $G$ of these quaternion algebras is equivalent to a multiplication of two  $2\times2$-matrices over their underlying ring. For example,  a quaternion $F=f_0+ f_1 i+f_2 j+f_3 k\in \mathbb{A}$ is isomorphic to a $2\times2$-matrix in $M_2(R')$ with generators
$$i= \left[ \begin{array}{cc}
                        1  & 0 \\
                       0  & -1
                       \end{array}\right],\qquad
 j= \left[ \begin{array}{cc}
                        0  & 1 \\
                       1  & 0
                       \end{array}\right].$$
Therefore, since $ij=k$, we obtain
\begin{eqnarray*}
F=f_0+ f_1 i+f_2 j+f_3 k&\cong& f_0\left[ \begin{array}{cc}
                        1  & 0 \\
                       0  & 1
                       \end{array}\right] +f_1 \left[ \begin{array}{cc}
                        1  & 0 \\
                       0  & -1
                       \end{array}\right]+f_2\left[ \begin{array}{cc}
                        0  & 1 \\
                       1  & 0
                       \end{array}\right]+f_3 \left[ \begin{array}{cc}
                        0  & 1 \\
                       -1  & 0
                       \end{array}\right]\\
                       &\cong& \left[ \begin{array}{cc}
                        f_0+f_1  & \quad f_2+f_3 \\
                        f_2-f_3  & \quad f_0-f_1
                       \end{array}\right].
\end{eqnarray*}
Strassen  \cite{Strassen} shows  that  two  $2\times2$  matrices  can  be  multiplied using  only  seven multiplications.
An improvement on Strassen's algorithm for $2\times2$  matrices is given by Winograd \cite{Winograd}. More information about Strassen's algorithm and its variants is available in \cite{shokrollahi} (page 11-13).
Hence, the multiplication of two quaternions $F$ and $G$ of $\mathbb{A}$ is equivalent to a multiplication of their equivalent  $2\times2$-matrices over the the ring $R'$. Thus, this multiplication is possible with 7 convolutional multiplications in $R'$ and some additional sums.

\textcolor{mycolor}{In the following we present some lemmas and a theorem that are needed for the key generation of our cryptosystem. The proof of some lemmas is trivial and they are left to the reader.   Without loss of generality, we assume that the points are ordered such that the first $|T|$ points are exactly the points of the subset $T$.}
\begin{lemma}\label{lemma3}
Let $\rho:\mathbb{A}_q\rightarrow (\mathbb{L}_q)^{n^2}$ be a map defined by
 \begin{eqnarray*}
\rho(F)&=& (F(a_1, b_1), F(a_2, b_2), \ldots, F(a_{n^2}, b_{n^2})),
\end{eqnarray*}
where $(a_i, b_i)\in E$ and \textcolor{mycolor}{$\mathbb{L}_{q}=\Big(\frac{1, 1}{\mathbb{Z}_q}\Big)$} is the algebra of quaternions over \textcolor{mycolor}{$\mathbb{Z}_q$.}
Then the following properties hold
\begin{eqnarray*}
 \rho(F+G)&=&\rho(F)+\rho(G),\\
\rho(F\circ G)&=&\rho(F)\bullet \rho(G),
\end{eqnarray*}
where $$\rho(F)\bullet \rho(G)=(F(a_1, b_1)\circ G(a_1, b_1), F(a_2, b_2)\circ G(a_2, b_2), \ldots, F(a_{n^2}, b_{n^2})\circ G(a_{n^2}, b_{n^2})).$$
\end{lemma}
\begin{proof}
Considering Equation (\ref{linear combination}) and (\ref{lambda}), the following relationship holds:
\begin{eqnarray*}
F(x, y)&=&\sum_{i=1}^{n^2}F(a_i,b_i)\circ \lambda_{a_i,b_i}(x,y),\\
F\circ G&=&\Big(\sum_{i=1}^{n^2}F(a_i,b_i)\circ \lambda_{a_i,b_i}(x,y)\Big)\circ G
=\sum_{i=1}^{n^2}F(a_i,b_i)\circ G \circ\lambda_{a_i,b_i}(x,y)\\
&=&\sum_{i=1}^{n^2}\big( F(a_i,b_i)\circ G(a_i,b_i)\big)\circ \lambda_{a_i,b_i}(x,y).
\end{eqnarray*}
\textcolor{mycolor}{We observed that since $\lambda_{a_i,b_i}(x,y)\in \mathbb{Z}_q[x,y]$ commutes with $G$, the second equation holds.}
Therefore,
\begin{eqnarray*}
\rho(F\circ G)&=&(F(a_1, b_1)\circ G(a_1, b_1), F(a_2, b_2)\circ G(a_2, b_2), \ldots, F(a_{n^2}, b_{n^2})\circ G(a_{n^2}, b_{n^2}))\\
&=&\rho(F)\bullet \rho(G).
\end{eqnarray*}
$\hfill \square $
\end{proof}
 \begin{lemma}\label{lemma2}
Let $\mathbb{A}_q=R'_q+R'_q i+R'_q j+R'_q k$ and let $\mathbb{B}_q=\mathbb{Z}_q[x,y]+\mathbb{Z}_q[x,y] i + \mathbb{Z}_q[x,y] j+ \mathbb{Z}_q[x,y] k$. Then
$$\mathbb{A}_q\cong \dfrac{\mathbb{B}_q}{\langle x^n -1,y^n -1\rangle} =\dfrac{\mathbb{Z}_q[x,y]+\mathbb{Z}_q[x,y] i + \mathbb{Z}_q[x,y] j+ \mathbb{Z}_q[x,y] k}{\langle x^n -1,y^n -1\rangle}.$$
\end{lemma}
\begin{lemma}
Let  $J_q=Q_q+Q_q i+Q_q j+Q_q k$, where $Q_q$ is the ideal of $R'_q$ generated by $\sigma=\sum\nolimits_{i=1}^{|T|} q_i \lambda_{a_i, b_i}(x,y)$.
Then $J_q$ is an ideal of $\mathbb{A}_q$ and is generated by $\sigma$. Furthermore, $J_q$ contains the ideal $\langle x^n -1,y^n -1 \rangle$ of $\mathbb{B}_q$.
\end{lemma}

In a similar way, we obtain the following lemma.
\begin{lemma}
Let $J=Q+Q i+Q j+Q k$, where $Q=\langle\sigma, q\rangle$ is an ideal of $R'$.  Then,  $ J$ is an ideal of $\mathbb{A}=R'+R' i+R' j+R' k$ and is generated by $q$ and $\sigma$, that is $J=\langle q, \sigma\rangle_{\mathbb{A}}$.
\end{lemma}
\begin{lemma}\label{lemma4}
Let $J=Q+Q i+Q j+Q k$ and $\mathbb{A}=R'+R' i+R' j+R' k$. Then
$$ \dfrac{\mathbb{A}}{J}\cong  \dfrac{R'}{Q}+ \dfrac{R'}{Q} i+\dfrac{R'}{Q} j+\dfrac{R'}{Q} k.$$
\end{lemma}
\begin{corollary}
The computation of  $F= f_0+ f_1 i+f_2 j+f_3 k \in \mathbb{A}$ modulo $ J$ is equivalent to the computation of $ f_i$ modulo $Q$, for   $i=0, \ldots, 3$.
\end{corollary}

In the following, we recall the combinatorial nullstellensatz theorem.
\begin{theorem}[Combinatorial Nullstellensatz]
Let $F$ be an arbitrary field, and let $f=f(x_1,\\ \ldots, x_n)$ be a polynomial in $F[x_1, \ldots, x_n]$. Suppose the degree of $f$ is $\sum\nolimits_{i= 1}^n t_i$, where each $t_i$ is a nonnegative integer and suppose the coefficient of  $\prod\nolimits_{i= 1}^n x_i^{t_i}$  in $f$ is nonzero. Then, if $U_1, \ldots, U_n$ are subsets of $F$ with $\mid U_i\mid > t_i$, there are $a_1 \in U_1, a_2\in U_2, \ldots, a_n\in U_n$ so that $f(a_1, \ldots, a_n)\neq 0$.
\end{theorem}
\begin{remark}
In accordance with the Combinatorial Nullstellensatz theorem \cite{Combinatorial-Nullstellensatz}, for any set of $T \subseteq E$, there is a unique polynomial $\sigma(x,y)$ that generates $Q$.
\begin{proof}
Let $T'$ and $T$ be two subsets of $E$ that generate the same polynomial $\sigma(x,y)$. Without loss of generality assume that the points are ordered such that the first $|T|$ (or $|T'|$) points are exactly the points of the subset $T$ (or $T'$). Equation (4) implies that 
 $\sigma(x,y)=\sum\nolimits_{i=1}^{|T|} q_i\lambda_{a_i, b_i}(x,y)=\sum\nolimits_{i=1}^{|T'|} q_i\lambda_{a_i, b_i}(x,y)$ and $k(x, y)=\sum\nolimits_{i=1}^{|T|} q_i\lambda_{a_i, b_i}(x,y)- \sum\nolimits_{i=1}^{|T'|} q_i\lambda_{a_i, b_i}(x,y)=0$. According to the Combinatorial Nullstellensatz theorem, since the degree of $k$ is $2(n-1)$ and $U_1=U_2=\{b\in \mathbb{Z}_q ~|~b^n=1\}$, where the condition $n |(q-1)$ yields $\mid U_1\mid=\mid U_2\mid=n $, there are $a \in U_1, b\in U_2$ so that $k(a, b)\neq0$. However $k(x,y)=0$, and this contradiction completes the proof. $\hfill \square $
\end{proof}
\end{remark}

In this cryptosystem, we  choose a set $T \subseteq E$ such that the cardinality of this set is small. Then, we construct the ideal $Q=\langle q, \sigma(x, y)\rangle$ of $R'$.
We consider  $J=Q+Q i+Q j+Q k$ as the \emph{private ideal} of $\mathbb{A}$ and use it for key generation and decryption. In the next subsection, we explain how to choose the subset $T$.

By considering a monomial basis for $R'$,  we get that $R'$ is isomorphic to $\mathbb{Z}^{n^2}$ as an additive \textcolor{mycolor}{modulo over $\mathbb{Z}$.} Since $Q$ is an ideal of $R'$, then $Q \subset R' \cong \mathbb{Z}^{n^2}$ is an additive subgroup of $\mathbb{Z}^{n^2}$ and so  $Q$ can be viewed as an  $n^2 $ dimensional lattice, denoted by $\Lambda_Q$, in $\mathbb{Z}^{n^2}$. Similarly, when we specify a basis for $\mathbb{A}$, we get that $\mathbb{A}$ is isomorphic to $\mathbb{Z}^{n^2}+\mathbb{Z}^{n^2} i+\mathbb{Z}^{n^2} j+\mathbb{Z}^{n^2} k$ that is also isomorphic to $\mathbb{Z}^{4n^2}$ as an additive group. The  \emph{private ideal} $J$ is an ideal of $\mathbb{A}=R'+R' i+R' j+R' k$ and $J \subset \mathbb{A} \cong \mathbb{Z}^{4n^2}$ can be viewed also as a lattice in $\mathbb{Z}^{4n^2}$. We use this private lattice, denoted by $\Lambda_{private}$, in our cryptosystem.

To construct a generator matrix for the private lattice $\Lambda_{\textrm{private}}$, we first give a generator matrix for the lattice $\Lambda_Q$ and develop it for $\Lambda_{\textrm{private}}$. In the following, we give the generator matrix of the lattice $\Lambda_Q$.

Let $D'$ be the generator matrix of $\Lambda_Q$.
Then, $D'$ is an $n^2 \times n^2$ matrix  whose rows are coefficient vectors of the Lagrange interpolants $\{\lambda_{a_{i}, b_{i}}\}_{i=1}^{|T|}$ and some $q e_{j_i}$'s to be explained next, where $j_i \in S \subseteq \{1, \ldots, n^2\}$, for $i=1, \ldots, n$ and  $e_k$ is the vector with one in $k$-th position and zero in other positions. Consider the matrix $L$ whose rows are the coefficient vectors of the Lagrange interpolants  $\{\lambda_{a_1, b_1},\ldots, \lambda_{a_{|T|}, b_{|T|}}\}$. Using that the row rank of an integer matrix is equal to its column rank  \cite{rank} (see Chapter 4), the row rank and column rank of $L$ is $|T|$.  We can assume that the columns $i_1, \dots, i_{|T|}$ are independent columns in this matrix  $L$.  Let
\begin{eqnarray*}
 S:=\{ j_1, \dots,j_{n^2-|T|}\}=\{1, \ldots, n^2\}\setminus\{i_1, \dots, i_{|T|}\},
 \end{eqnarray*}
where $ j_1 < j_2 < \cdots < j_{n^2-|T|}$. Then, $q e_{j_1} \dots, q e_{j_{n^2-|T|}}$ and the coefficient vectors of interpolants  $\{\lambda_{a_{i}, b_{i}}\}_{i=1}^{|T|}$  form a basis for $\Lambda_Q$ with high probability.
Now, since $J\cong Q^{4}$ as  additive groups, the private lattice $\Lambda_{\textrm{private}}$ is generated by the rows of the following matrix

\begin{eqnarray}\label{private lattice}
M_{\textrm{private}}=\left[
                   \begin{array}{cccc}
                      D' & 0 & 0 & 0\\
                      0 & D' & 0 & 0\\
                      0 & 0 & D' & 0\\
                      0 & 0 & 0 & D'
                       \end{array}
                     \right].
\end{eqnarray}
\subsection{ Key generation}
For creating public and private keys, we randomly choose two quaternion polynomials $F$ and $G$ in $\mathbb{A}$ as
\begin{eqnarray*}
F&=&f_0+ f_1 i+f_2 j+f_3 k,  \quad f_0,f_1, f_3, f_4 \in L_f,\\
G&=&g_0+g_1 i+g_2 j+g_3 k, \quad g_0,g_1, g_3, g_4 \in L_g.
\end{eqnarray*}
The polynomials  $f_i\in L_f$ and  $g_i\in L_g$, for $i=0, \ldots, 3$, are small ternary polynomials in $R'$, that is,
they are random polynomials that most coefficients are zero and the rest are in the set $\{1, -1\}$. Therefore, the two quaternion polynomials  $F$ and $G$ are small ternary quaternion polynomials.  If we consider the coefficient vector of each element $F=f_0+ f_1 i+f_2 j+f_3 k \in \mathbb{A}$ as
 \begin{eqnarray*}
\overrightarrow{F}&=&(\overrightarrow{f_0}, \overrightarrow{f_1}, \overrightarrow{f_2}, \overrightarrow{f_3})
  \end{eqnarray*}
then, the coefficient vectors of $\overrightarrow{F}$ and $\overrightarrow{G}$ have many zeros and few $1$ and $-1$. Therefore, the Euclidean norm of $\overrightarrow{F}$ and $\overrightarrow{G}$ is small.

Both  $F$ and $G$  should be invertible in $\mathbb{A}/J$, that is, there exist $F^{-1}$ and $G^{-1}$ in $\mathbb{A}$ such that
\begin{eqnarray*}
  F\circ F^{-1}&=&F^{-1}\circ F= 1 \; \;(\mbox{mod } J),\\
   G\circ G^{-1}&=&G^{-1}\circ G= 1 \;\;(\mbox{mod } J).
\end{eqnarray*}
The  quaternion  $F$  must  satisfy  the  additional  requirement that  it  is invertible in $\mathbb{A}_p$ for decryption purpose. In order for the quaternion $F$ to be invertible over $\mathbb{A}_p$, the polynomial \textcolor{mycolor}{$N(F)=f_0^{2}-f_1^{2}-f_2^{2}-f_3^{2}$} must be nonzero and invertible over its underlying ring $R'_p$. By Lemma \ref{lemma4},  $\mathbb{A}/J$ is a quaternion algebra over the ring  $R'/Q$. Hence, the necessary and sufficient condition for invertibility of  $F$ and $G$ in $\mathbb{A}/J$ is the existence of $(N(F))^{-1}$ and  $(N(G))^{-1}$ in the ring $R'/Q$.

In the sequel, we present a method for choosing $F$ and $G$ to fulfill the above conditions. Our method uses the fact that the polynomial $N(F)$ is invertible modulo $Q$ if and only if the roots of $N(F)$ which belong to $E$ are also contained in $T$. Thus our method selects  $T$ after choosing $F$ and $G$ to make sure that $N(F)$ and $N(G)$ are invertible modulo $Q$:  choose $G$ at random and define the set $T=\bigcap\nolimits_{i=0}^3\{(a, b) \in E ~|~ g_i(a, b)=0 \}$. If $T=\emptyset$, then pick another $G$. Select $F$ randomly such that $\{(a, b) \in E ~|~  N(F)(a, b)=0 \}\subset T $.

The inverses of the quaternion $F$ in $\mathbb{A}/J$ and $\mathbb{A}_p$ are denoted by $F^{-1}$ and $F^{-1}_{p}$, respectively; they are computed as follows
\begin{eqnarray*}
F^{-1}&=&(f_0^{2}+f_1^{2}+f_2^{2}+f_3^{2})^{-1} \circ \bar{F}=l_0+l_1 i+l_2 j+l_3 k;\quad \quad l_0,l_1,l_2,l_3 \in R'/Q \\
F^{-1}_{p}&=&(f_0^{2}+f_1^{2}+f_2^{2}+f_3^{2})^{-1} \circ \bar{F}=s_0+s_1 i+s_2 j+s_3 k; \quad \quad s_0,s_1,s_2,s_3 \in R'_p.
\end{eqnarray*}
A necessary issue in public key generation is to choose a quaternion $W=w_0 +w_1\cdot i+w_2\cdot j+w_3\cdot k$ \textcolor{mycolor}{in $\mathbb{L}_{q}=\Big(\frac{1, 1}{\mathbb{Z}_q}\Big)$, such that $W$ be invertible over $\mathbb{L}_{q}$ and $w_0, w_1, w_2, w_3\neq 0$.} This quaternion prevents the leakage of public key information.
Then we compute $\tilde{H}=F^{-1} \circ G + \vartheta \pmod q$ where $F\circ F^{-1}=F^{-1}\circ F= 1\pmod J$ and $\vartheta=W\circ \sigma \in J_{q}$ is unknown to the attacker. The quaternion
\begin{eqnarray}\label{public key}
\tilde{H}&=&F^{-1} \circ G + \vartheta\;\;(\mbox{mod } q),
\end{eqnarray}
is the public key in our cryptosystem. \textcolor{mycolor}{The modulo $q$ operation is componentwise, it means that the corresponding quaternion of $F^{-1} \circ G + \vartheta$ in $\mathbb{A}_q$ that is $(h_0 (\bmod q))+(h_1 (\bmod q))i+(h_2(\bmod q))j+(h_3 (\bmod q))k$ is considered as the public key $\tilde{H}$.} \textcolor{mycolor2}{The coefficient vectors $\overrightarrow{h_i(\bmod q)}$'s, for $i=0, \ldots, 3$, denote the public key $\tilde{H}$.  Therefore, the memory needed to store the public key $\tilde{H}$ is equal to $4n^2\lceil\log_2 q\rceil$ bits.}

We have that $F\circ F^{-1}=F^{-1}\circ F= 1 +\alpha\circ \sigma\pmod q$, where $\alpha \in \mathbb{A}$  is unknown to the attacker. Therefore, multiplying Equation (\ref{public key}) by $F$ implies that
\begin{eqnarray*}
F\circ \tilde{H} &=&F\circ F^{-1} \circ G + F\circ\vartheta\;\;(\mbox{mod } q) \\
&=& (1+\alpha\circ \sigma)\circ G +F\circ\vartheta\;\;(\mbox{mod } q)\\
&=& G+\sigma\circ \alpha\circ G +F\circ\vartheta\;\;(\mbox{mod } q)\\
&=& G+\gamma\;\;(\mbox{mod } q).
\end{eqnarray*}
Thus,
\begin{eqnarray}
F\circ \tilde{H} &=& G+\gamma\;\;(\mbox{mod } q), \label{public equation}\\
\gamma&=&\sigma \circ\alpha\circ G+F\circ\vartheta, \label{gamma}
\end{eqnarray}
where $\gamma \in J_{q}$.
Subsequently, the private key consists of $(\overrightarrow{G}, \overrightarrow{F}, -\overrightarrow{\rho(\gamma)}) $, where $\overrightarrow{\rho(\gamma)}=\big(\rho(\gamma_0), \ldots, \rho(\gamma_3)\big)$.

Using Lemma \ref{lemma3}, we have
\begin{eqnarray}
\rho(\sigma)&=&\left(q_1, \ldots, q_{|T|}, 0,\ldots, 0\right),  \\
\rho(G)&=&\left(0, \ldots, 0, G(a_{|T|+1}, b_{|T|+1}),\ldots, G(a_{n^2}, b_{n^2})\right),\label{G}\\
\rho(\vartheta)&=&\left(q_1W, \ldots, q_{|T|}W, 0, \ldots, 0\right). \label{vartheta}
\end{eqnarray}
Therefore,
\begin{eqnarray}\label{eq3}
\rho(\gamma)&=&\rho(\sigma\circ \alpha\circ G)+\rho(F\circ \vartheta)\nonumber \\
&=&(0,\ldots, 0)+\rho(F\circ \vartheta)\nonumber\\
&=&\rho(F\circ \vartheta).
\end{eqnarray}
\subsection{ Encryption}
To encrypt the incoming data $m$, we  convert it into a quaternion having four small polynomials.
We map $m$ to a quaternion $M=m_0+m_1 i+m_2 j+m_3 k ~\in \mathbb{A}$, where $m_{i} \in L_m$, for $i=0, \ldots, 3$. The incoming data can be generated from the same or four different sources. Then, the data is transformed into the  quaternion $M$.  After that, we pick a random quaternion  $\Phi=\phi_0+\phi_1 i+\phi_2 j+\phi_3 k $ such that $\phi_i \in L_\phi$, for $i=0, \ldots, 3$.
The encrypted message is
\begin{eqnarray*}
C=p\tilde{H}\circ \Phi+M\;\;(\mbox{mod } q).
\end{eqnarray*}
In this cryptosystem we can set up to four messages at once  and encrypt them simultaneously as one quaternion and transmit it as the ciphertext. Encryption and key generation need one quaternion multiplication including 16 convolution multiplications but we show that the multiplication of two quaternions of $\mathbb{A}_q$ is equivalent to a multiplication of their equivalent  $2\times2$-matrices over the the ring $R'_q$.  Moreover,  by Strassen method \cite{Strassen} the multiplication of two quaternions is possible with only 7 convolutional multiplications.

Since each quaternion in our cryptosystem is a $4$-tuple vector of integers and since we use bivariate polynomials, an instance of BQTRU with the parameters $(n,p,q)$ is comparable with an instance of NTRU with the public parameters $(4n^2,p,q)$. We believe this is a fair comparison between NTRU and our system.

We note that the number of multiplication operations in a convolutional product of two polynomials using standard arithmetic with degree $n-1$ is equal to $n^2$. Therefore, if we consider NTRU with the public parameters $(4n^2,p,q)$, the number of multiplication operations in encryption and key generation is  $(4n^2)^2=16n^4$. While,  the multiplication of two quaternions in our encryption and key generation process needs only 7 convolutional multiplications in $R'_q$. Therefore, the number of multiplication operations in key generation and encryption of our cryptosystem with dimension $n$ is $7n^4$. Hence, if we apply Strassen method in our cryptosystem, the key generation and encryption process is  $16/7$ times faster than NTRU with dimension comparable parameters.

\subsection{ Decryption}
 Let $C \in \mathbb{A}_q$ be the received ciphertext. In  order  to  decrypt  it,  a receiver first  computes
\begin{eqnarray*}
F\circ C\;\;(\mbox{mod } q).
\end{eqnarray*}
Then she  finds the closest vector $B$ to $F\circ C$ in the private lattice $\Lambda_{\textrm{private}}$. If we choose the public parameters carefully, then $V=F\circ C-B$  is a  quaternion with  integer  coefficients, and the receiver  recovers  the  message by  computing $$F^{-1}_{p} \circ V \;\;(\mbox{mod } q).$$
The  decryption  works since
\begin{eqnarray*}
F\circ C&=&F\circ(p\tilde{H}\circ \Phi+M)\;\;(\mbox{mod } q)\\
&=&p F\circ\tilde{H}\circ \Phi+F\circ M \;\;(\mbox{mod } q)\\
&=&p (G+\gamma)\circ \Phi+F\circ M\;\;(\mbox{mod } q)\\
&=&p G\circ \Phi+p \gamma\circ \Phi+F\circ M\;\;(\mbox{mod } q)\\
&=&p G\circ \Phi+F\circ M+(p \gamma\circ \Phi+\varepsilon q)\\
&=&p G\circ \Phi+F\circ M+B,
\end{eqnarray*}
\textcolor{mycolor}{where $\varepsilon \in \mathbb{A} $ and $B=(p\alpha\circ G\circ \Phi+p F\circ W\circ \Phi)\circ\sigma+\varepsilon q \in J $ is unknown to the receiver.} The quaternion $p G\circ \Phi+F\circ M \in \mathbb{A}_q$ should be found in order to make the process feasible.  To this end, we find $B\in J$ and after that we subtract it from $F\circ C=p G\circ \Phi+F\circ M+B$. Then, the process continues in the same fashion as in the NTRU decryption method.

Considering the vector representation, we have $\overrightarrow{F\circ C}=\overrightarrow{p G\circ \Phi+F\circ M}+\overrightarrow{B}$.
We show that $\overrightarrow{B}\in \Lambda_{\textrm{private}}$, is the closest vector of $\Lambda_{\textrm{private}}$ to $\overrightarrow{F\circ C}$: the vector  $\overrightarrow{B}$ is a integer linear combination of the rows of the generator matrix $M_{\textrm{private}}$ which consist of some of $\overrightarrow{\lambda_{a, b}}$'s.  We observe that, the polynomials $\lambda_{a, b}$ have many monomials with large coefficients. So,  the Euclidean norm of the vectors  $\overrightarrow{\lambda_{a, b}}$ and $\overrightarrow{B}$ (as integer linear combination of  $\overrightarrow{\lambda_{a, b}}$'s) is very large.
Moreover, since $F$, $G$, $\Phi$ and $M$ are small ternary quaternion polynomials, that is, they have few coefficients that are equal to $1$ and $-1$, the Euclidean norm of   $\overrightarrow{F}, \overrightarrow{G}, \overrightarrow{\Phi}$ and $\overrightarrow{M}$ is small. Therefore, the Euclidean norm $\overrightarrow{p G\circ \Phi+F\circ M}$ is very small relative to the Euclidean norm of $\overrightarrow{B}$. Thus, treating $\overrightarrow{p G\circ \Phi+F\circ M}$ as a noise, $\overrightarrow{F\circ C}$ can be viewed as a perturbed lattice point $\overrightarrow{B}$ with small noise vector and $\overrightarrow{B}$ is,  with high probability, the closest vector of $\Lambda_{\textrm{private}}$ to $\overrightarrow{F\circ C}$.

Due to the  structure of the private lattice $\Lambda_{\textrm{private}}$  and the probabilistic property of the noise vector, it is easy to find the CVP in $\Lambda_{\textrm{private}}$. More information about this can be found in Section \ref{Successful Decryption}.

If we choose suitable public parameters for our cryptosystem, the coefficients of the four polynomials in $p G\circ \Phi+F\circ M$ lie in the interval $[-\frac{q-1}{2}, \frac{q-1}{2}]$ and decryption does not fail. Then, we can reduce $p G\circ \Phi+F\circ M$ modulo $p$ and the term $p G\circ \Phi$ vanishes. Hence, the message $M$ is recovered by multiplying $F\circ M \pmod p$ by $F^{-1}_{p}$ from the left side.


The decryption phase of our cryptosystem, with dimension $n$, needs to calculate 14 convolution products instead of 32 in NTRU cryptosystem  with dimension $4n^2$. The  legitimate receiver should solve a CVP in the private lattice $\Lambda_{\textrm{private}}$ and this makes the decryption process of our cryptosystem to be slower than  NTRU. However, the presence of the private ideal makes our cryptosystem to be secure against lattice based attacks; this is explained in Section \ref{Lattice attacks}.
 \section{Successful decryption}\label{Successful Decryption}
It is important that the public parameters  in the cryptosystem be chosen  in such a way that the probability of decryption failure is very small (e.g., $2^{-80}$) \cite{book}.
The decryption process is successful if the closest vector to $F\circ C=p G\circ \Phi+F\circ M+B$ is found correctly and all quaternion coefficients of $V=p G\circ \Phi+F\circ M$ lie in the interval $[-\frac{q-1}{2}, \frac{q-1}{2}]$.

Kouzmenko  \cite{NTRUG} demonstrates the validity of a simple probabilistic model for the coefficients of the corresponding polynomial in the NTRU cryptosystem. \textcolor{mycolor}{Using  this model, we estimate the probability distribution of the coefficients of the polynomials in the quaternion $V=v_0+v_1i+v_2j+v_3k$, under certain reasonable hypotheses.}
Therefore, by using similar computations as in \cite{QTRU}, we have
\begin{eqnarray}
\textrm{E}(v_{i, j})&=&0,\\
\textrm{Var}(v_{i, j})&=&\frac{16 p^2 d_\phi d_g}{n^2}+\frac{4 d_f (p^2-1)}{6}, \label{variance}
\end{eqnarray}
for $i=0, \ldots, 3$ and $j=0, \ldots, n^2-1$.
This together with the assumption that the $v_{i, j}$s have Gaussian distribution with zero mean and standard deviation $\theta=\sqrt{\frac{16 p^2 d_\phi d_g}{n^2}+\frac{4 d_f (p^2-1)}{6}}$, gives
\begin{eqnarray*}
Pr\Big(|v_{i,j}| \leq \frac{q-1}{2}\Big)=2\Phi\Big(\frac{q-1}{2\theta}\Big)-1,
\end{eqnarray*}
for $i=0, \ldots, 3$ and $j=0, \ldots, n^2-1$,  where $\Phi$ is the normal cumulative distribution function.
Thus, the probability for all coefficients of $V$ lying in the interval $[-\frac{q-1}{2}, \frac{q-1}{2}]$, is determined as
$$\Big(2\Phi\Big(\frac{q-1}{2\theta}\Big)-1\Big)^{4n^2}.$$

Another reason  for unsuccessful decryption is failing to estimate the vector $B$ from $V':=F\circ C=V+B$. Finding $B$ involves solving a CVP in the private lattice $\Lambda_{\textrm{private}}$. In general, CVP is an NP-hard problem and  the computational complexity of the algorithms for finding the closest vector of a lattice to an arbitrary vector is intractable, however it is tractable in our special case which is described in the sequel.

We note that the  closest vector problem is equivalent to the integer least-squares (ILS) problem which is the problem of finding the least-squares solution to a system of linear equations where the unknown vector is comprised of integers, but the matrix coefficient and given vector are comprised of real numbers.  Lifting $V'=v'_0+v'_1i+v'_2j+v'_3k$ from $\mathbb{A}/J$ into $\mathbb{A}$ involves solving  the CVP  problem or the ILS problem in $\Lambda_{\textrm{private}}$ with the generator matrix $M_{\textrm{private}}$ presented in Equation (\ref{private lattice}):
\begin{eqnarray}\label{ILS}
\underset{z\in \mathbb{Z}^{4n^2}}{\textrm{Argmin}} \|\overrightarrow{V'}-zM_{\textrm{private}} \|^2 = \sum\limits_{i=0}^3\underset{z_i\in \mathbb{Z}^{n^2}}{\textrm{Argmin}} \|\overrightarrow{v'_i}-z_iD' \|^2.
\end{eqnarray}
As argued in the previous section, the ideal $Q$ can be viewed as  an $n^2$ dimensional lattice $\Lambda_Q$ generated by the matrix $D'$. Therefore,
this step of our decryption involves solving four CVP's in dimension $n^2$. Indeed, our problem is equivalent to finding the closest lattice point in $\Lambda_Q$ to the given point $\overrightarrow{v'_i}$ in Equation (\ref{ILS}), for $i=0, \ldots, 3$, that has been perturbed by an additive noise vector $\overrightarrow{v_i}$, where $(\overrightarrow{v_0}, \overrightarrow{v_1}, \overrightarrow{v_2}, \overrightarrow{v_3})=\overrightarrow{V}=\overrightarrow{p G\circ \Phi+F\circ M}$.

Our highest level of security occurs in $n\geq 11$ ($n^2\geq 121$). There exist less complex methods  for computing CVP in such small dimensions. One of the prominent methods is sphere decoding that searches for the closest lattice point within a given hyper sphere \cite{Viterbo}, \cite{Hassibi}.
The sphere decoding is feasible  in small dimensions, typically in dimensions less than 130 \cite{Sole}. By applying some techniques (like LLL lattice reduction algorithm), the complexity of this algorithm can be decreased \cite{Vardy}.

We observe that $\overrightarrow{V}$ has Gaussian distribution with zero mean and variance given in Equation (\ref{variance}). We show in a numerical example that the variance of this additive noise is small for an optimal decoder like sphere decoder. In the rest of this subsection, it is discussed  that in our case, the expected complexity of the sphere decoding algorithm  is polynomial in the dimension of the lattice $\Lambda_Q$. Furthermore, since the dimension of  $\Lambda_Q$ in the highest security level is $121$, which is less that $130$, solving the CVP problems in our decryption process are feasible.

We use the Poltyrev limit \cite{Poltyrev} to know the amount of the maximum variance of the additive Gaussian noise under which a  maximum-likelihood decoder like sphere decoder  results in a correct estimation of a desired lattice vector. Indeed,  Poltyrev proved that there exists a lattice $\Lambda$, with basis matrix $L$ and large enough dimension $N$, such that decoding is possible with arbitrarily small error probability if and only if the variance of the additive noise vector is smaller than
$$\sigma^2_{\textrm{max}}:=\frac{\big(\textrm{vol}(\Lambda)\big)^{\frac{2}{N}}}{2\pi e},$$
where $\textrm{vol}(\Lambda)=|\det L|$.

If a lattice point is perturbed by a Gaussian noise with zero mean and variance $\sigma^2$, which is smaller than  $\sigma^2_{\textrm{max}}$, the CVP problem can be solved in polynomial (or even cubic) time using sphere decoding algorithm  \cite{Hassibi}.
Indeed, the error probability of the sphere decoding algorithm and its complexity are related to the variance of the additive noise.  When the variance of the additive noise is much smaller than $\sigma^2_{\textrm{max}}$, the error probability  of the sphere decoding algorithm  approaches zero and the algorithm returns the correct estimation for a desired lattice vector. In this case, the complexity of the sphere decoding algorithm is polynomial in the dimension of the lattice, which is reasonable and affordable in small dimensions (typically less than 130).

According to the basis of $\Lambda_Q$, $\textrm{vol}(\Lambda_Q)=|\det D'|> q^{n^2-|T|}$ and a lower bound for  $\sigma^2_{\textrm{max}}$ is estimated as
\begin{eqnarray*}
\sigma^2_{\textrm{max}}&>&\dfrac{q^{\frac{2(n^2-|T|)}{n^2}}}{2\pi e}=\dfrac{q^{2\big(1-\frac{|T|}{n^2}\big)}}{2\pi e}.
\end{eqnarray*}
We have assumed that $|T|\leq n$, therefore $\sigma^2_{\textrm{max}}> \frac{q^{2\big(1-\frac{1}{n}\big)}}{2\pi e}$. For the dimension $n=11$, the highest level of security in our cryptosystem,  $\sigma^2_{\textrm{max}}>885$, while the variance of the additive noise, $\textrm{Var}(v_i)\approx 263$, for  $i=0, \ldots, 3$.

Since  $|T|$ is chosen very small compared to $n^2$, the variance of the vectors  $v_i$'s is much smaller than  $\sigma^2_{\textrm{max}}$. Also, the dimension of the lattice $\Lambda_Q$ in our highest level of security is $n^2=121< 130$.
Hence, if a legitimate receiver which has generator matrix $D'$, uses the sphere decoding algorithm, then
solves the CVP in the decryption process with negligible  error probability and affordable complexity  and continues the decryption.
\subsection{Parameter selection and comparisons}\label{Parameter Selection}
In BQTRU cryptosystem, we consider $p\neq 2$ since the probability of invertibility $F$ modulo $p=2$ is too low. The initial condition $n |(q-1)$ is required. Since a primitive $k$-th root of unity does not always exist modulo $q$, it is necessary to work in an extension ring. The condition $n |(q-1)$ yields the factorization of  $x^n -1$ into linear factors in $\mathbb{Z}_q$. Then, $E$ is a set of elements of $\mathbb{Z}_q \times \mathbb{Z}_q$ which has maximum cardinality $n^2$. The subsets $T\subseteq E$ are chosen such that $|T|$ is small compared to $n^2$ (e.g. $|T|< n$). Therefore, the number of small subsets $T\subseteq E$ that can be selected is approximately equal to $\sum\nolimits_{i= 1}^n {n^2\choose i}$.
Then, the number of ideals $Q_q$ generated by $\sigma(x,y)=\sum\nolimits_{i=1}^{|T|} q_i\lambda_{a_i, b_i}(x,y)$ that can be used for decryption is
$$\sum\limits_{i= 1}^n (q-1)^{i}{n^2\choose i}.$$
For this reason, an exhaustive search for the ideal $Q_q$ containing $\langle x^n -1, y^n -1\rangle$ in $\mathbb{Z}_q[x, y]$ is infeasible when $n$ is large enough. The large number of these ideals  enhances the security of the cryptosystem against brute force attack.

\textcolor{mycolor}{The security of the NTRU cryptosystem depends on the size of all polynomials $f,g,r,m$ and $h$. Since finding the private key $(f,g)$ leads to solving SVP in the NTRU lattice  \cite{Coppersmith},  when the sizes of $f$ and $g$  increase ($d_f$ and $d_g$ increase), the security of the private key enhances. On the other hand, obtaining the message $m$ leads to finding CVP in the NTRU lattice   \cite{Message Attack} and then if the sizes of $m$ and $r$  increase ($d_m$ and $d_r$ increase),  the security of the message enhances. Therefore, increasing the size of $f,g,r$ and $m$  is reasonable for enhancing the security of NTRU but by increasing the size of $f,g,r$ and $m$,  the expected size of $p.g\ast r+f\ast m$ is increased and then the probability of decryption failure is relatively increased. Hence,  there is a tradeoff between the message security and the key security; we can not increase the size of the private key  $f, g$ (for the key security) and the polynomials $m$ and $r$ (for the message security), simultaneously.  }

In the proposed scheme, the presence of $Q$ allows us to reduce the size of $F$ and $G$ (that is, the size of $f_i$'s and $g_i$'s, for $i=0, \ldots, 3$) without detracting from the private key's security. This enables us to increase the sizes of $\Phi$ and $M$ (increase $d_\phi$ and $d_m$) without increasing the probability of decryption  failure. In this sense, the proposed scheme is more robust against message attacks than NTRU.

The decryption of NTRU is impractical when the sizes of the private keys are large while in BQTRU, the decryption with small sizes of $F$ and $G$ is impossible when $|T|$ is large.

The parameter sets in Table \ref{table1}, are chosen sufficiently large to eliminate the possibility of decryption failure. Namely, if we pick  $d_{f}=d_{g}=d_{\phi}=d_{m}=d$, since all polynomials in $V=p G\circ \Phi+F\circ M$ are ternary, their coefficients have maximum value equal to $2(6d+ 1)p$. It follows that no coefficient of  $V$ can exceed $2(6d+ 1)p$ in absolute value and the decryption is guaranteed to be successful if $2(6d+ 1)p<q/2$ or $q>4(6d+ 1)p$. Therefore, decryption is successful if $q>24dp$  approximately  and since $d\approx n^2/7$, then $q>24pn^2/7$ is sufficient condition for the successful decryption in our cryptosystem.
 However, an examination of the computation of the  maximum coefficient of  $V$  shows that decryption is likely to succeed even for considerably smaller values of $q$.  Therefore, similar to the NTRU cryptosystem, for additional efficiency and to reduce the size of the public key, it may be advantageous to choose a smaller value of $q$.

In the NTRU cryptosystem with parameters $(n, q', p)$, which was introduced  in Section \ref{NTRU}, we state that if $q'>(6d+ 1)p$ or $q'>2np$ (when we select $d\approx n/3$),  then decryption never fails. Hence for the NTRU cryptosystem with parameter sets $(4n^2, q', p)$, if  $q'> 8n^2p$, it has no decryption failure. It implies  that $q/q'\approx 3/7$. Therefore, choosing the prime number $q$
such that $n |(q-1)$ and $q\approx\frac{3}{7}q'$,  provides $(n, q, p)$-BQTRU  with an equivalent level of protection against decryption failure as $(4n^2, q', p)$-NTRU.
We summarize our results about the parameter's selection in Table \ref{table1} given at the end of Section \ref{Lattice attacks}.
\section{Security analysis}\label{Lattice attacks}
In this section, we discuss the key security and the message security of our cryptosystem.
We prove that the security of the proposed cryptosystem relies on the hardness of SVP in a certain type of  lattice.
We show that our cryptosystem is secure against lattice based attacks due to the specific structure of its lattice.

\textcolor{mycolor}{We recall that the convolution multiplication of two polynomials $f$ and $h$ in $R=\mathbb{Z}[x]/\langle x^{n}-1\rangle$ is equivalent to the matrix multiplication of the coefficient vector $\overrightarrow{f} $ and  the circulant matrix  $\mathcal{H}$  of the coefficient vector $\overrightarrow{h}$. This matrix is not circulant when $f$ and $h$ belong to $R'=\mathbb{Z}[x,y]/\langle x^{n}-1, y^{n} -1\rangle$. It can be shown that each row of $\mathcal{H}$ is a permutation of the coefficient vector $\overrightarrow{h}$ and these permutations are related to the monomial order that we have used for $R'$. Consequently,  the convolution product of two polynomials $f$ and $h$ in the proposed cryptosystem is also equivalent to the matrix multiplication $\overrightarrow{f}\mathcal{H}$, where  $\mathcal{H}$ is a matrix such that its rows are  a permutation of the coefficient vector $\overrightarrow{h}$.}

In the sequel,  we explain how the key recovery of our cryptosystem is formulated as a shortest vector problem in a certain special form of a lattice.

 The parameters $n$, $p$, $q$, $d_f$, $d_g$, $d_\phi$  and the public key
$\tilde{H}=F^{-1} \circ G + \vartheta$, where $F\circ F^{-1}=F^{-1}\circ F= 1\pmod J$,
are known to the attacker. For key recovery, the attacker should find a vector $(G, F, -\rho(\gamma)) $ which satisfies the equation  $F\circ\tilde{H}= G+\gamma\pmod q$.
Let $F=f_0+f_1 i+f_2 j+f_3 k$ and $ \tilde{H}=h_0+h_1 i+h_2 j+h_3k$, where $f_i$ and $h_i$ are polynomials in the ring $R'$, for $i=0, \ldots,3$. Then
\begin{eqnarray*}
F\circ \tilde{H}&=&(f_0+f_1 i+f_2j+f_3 k) \circ (h_0+h_1 i+h_2 j+h_3 k))\\
        &=&(f_0\ast h_0+f_1\ast h_1+f_2\ast h_2-f_3\ast h_3)+(f_0\ast h_1+f_1\ast h_0+f_2\ast h_3-f_3\ast h_2) i\\
       &+&(f_0\ast h_2-f_1\ast h_3+f_2\ast h_0+f_3\ast h_1) j+(f_0\ast h_3+f_1\ast h_2-f_2\ast h_1+f_3\ast h_0) k.
 \end{eqnarray*}
If we expand the quaternion equation $F\circ\tilde{H}= G+\gamma \pmod q$ for a given quaternion $\tilde{H}$, then we have a system of linear equations
\begin{eqnarray}\label{system}
\left\{ \begin{array}{l}
f_0\ast h_0+f_1\ast h_1+f_2\ast h_2-f_3\ast h_3-\gamma_{0}= g_0+q u_0,\\
f_0\ast h_1+f_1\ast h_0+f_2\ast h_3-f_3\ast h_2-\gamma_{1}= g_1+q u_1,\\
f_0\ast h_2-f_1\ast h_3+f_2\ast h_0+f_3\ast h_1-\gamma_{2}= g_2+q u_2,\\
f_0\ast h_3+f_1\ast h_2-f_2\ast h_1+f_3\ast h_0-\gamma_{3}= g_3+q u_3,
\end{array} \right.
\end{eqnarray}
 where $u_0, u_1, u_2, u_3 \in R'$.
We consider
$$
\mathcal{\tilde{H}}=\left[
\begin{array}{cccc}
\mathcal{H}_0  & \mathcal{H}_1 &  \mathcal{H}_2  & \mathcal{H}_3\\
\mathcal{H}_1 & \mathcal{H}_0 & -\mathcal{H}_3  & \mathcal{H}_2\\
\mathcal{H}_2  & \mathcal{H}_3  & \mathcal{H}_0 & -\mathcal{H}_1\\
-\mathcal{H}_3  & -\mathcal{H}_2 & \mathcal{H}_1 & \mathcal{H}_0
  \end{array}
  \right],$$
where $\mathcal{H}_i$ is the corresponding matrix of
$\overrightarrow{h_i}= [h_{i0}, h_{i1},\cdots ,h_{in^2-1}]$, for $i=0, \ldots,3$.
According to the above mentioned issues
\begin{eqnarray}\label{product equvalence}
\overrightarrow{F\circ \tilde{H}}&=&\overrightarrow{F}\mathcal{\tilde{H}}=
\left[ \begin{array}{cccc}
      \overrightarrow{f_0} & \overrightarrow{f_1} & \overrightarrow{f_2 } &  \overrightarrow{f_3}
\end{array}
 \right]
\left[
\begin{array}{cccc}
\mathcal{H}_0  & \mathcal{H}_1 &  \mathcal{H}_2  & \mathcal{H}_3\\
\mathcal{H}_1 & \mathcal{H}_0 & -\mathcal{H}_3  & \mathcal{H}_2\\
\mathcal{H}_2  & \mathcal{H}_3  & \mathcal{H}_0 & -\mathcal{H}_1\\
-\mathcal{H}_3  & -\mathcal{H}_2 & \mathcal{H}_1 & \mathcal{H}_0
  \end{array}
  \right]. \nonumber
\end{eqnarray}
Therefore, when we represent the elements of $\mathbb{A}$ by their coefficient vectors in $\mathbb{Z}^{4n^2}$, the set $\Lambda_{\textrm{BQTRU}}$ formed by the vectors $(\overrightarrow{G}, \overrightarrow{F}, -\overrightarrow{\rho(\gamma)}) $ that are also equivalent to
the set of all solutions $(G, F, \gamma) $ to Equations (\ref{system}), forms an integer lattice.
From the above equations one obtains that the BQTRU lattice, denoted by $\Lambda_{\textrm{BQTRU}}$, is generated by the following matrix
\begin{eqnarray}\label{matrix}
M_{\textrm{BQTRU}}=\left[\begin{array}{ccc}
                    qI_{4n^2} &  ~~0 &  ~~ 0 \\
                     \mathcal{\tilde{H}} &  I_{4n^2} & ~~  0\\
                     \mathfrak{D} & ~~ 0& ~  ~~ I_{4n^2}
                     \end{array}
                  \right],
\end{eqnarray}
where
$$ \mathfrak{D} =\left[\begin{array}{cccc}
                      D \quad & 0  \quad  & 0  \quad & 0\\
                      0  \quad & D \quad  & 0  \quad & 0\\
                      0  \quad & 0  \quad & D  \quad & 0\\
                      0  \quad & 0  \quad & 0  \quad & D
                       \end{array}\right].$$
 and $D $ is an ${n^2\times n^2} $ matrix which consists of $\overrightarrow{\lambda_{a, b}}$, coefficient vectors of all Lagrange interpolators $\lambda_{a, b}(x, y)$, such that $(a,b)\in E$.

We now consider that the public key $\tilde{H}$ is created using the private quaternions $F$,  $G$ and $\gamma$ and compute what happens when we multiply the $M_{\textrm{BQTRU}}$ by a carefully chosen vector.
\begin{proposition}
Assume that $F\circ\tilde{H}= G+\gamma\pmod q$, let $U\in \mathbb{A}$ be the quaternion that satisfies
\begin{equation}~\label{1}
F\circ\tilde{H}= G+\gamma +q U.
\end{equation}
Then
\begin{equation*}~\label{2}
(\overrightarrow{-U} , \overrightarrow{F}, -\overrightarrow{\rho(\gamma)})M_{\textrm{BQTRU}} = (\overrightarrow{G} , \overrightarrow{F}, -\overrightarrow{\rho(\gamma)}),
\end{equation*}
and therefore the vector $(\overrightarrow{G} , \overrightarrow{F}, -\overrightarrow{\rho(\gamma)})$ belongs to the BQTRU lattice.
\end{proposition}
\begin{proof}
When we multiply the vector $(\overrightarrow{-U} , \overrightarrow{F}, -\overrightarrow{\rho(\gamma)})$ by the columns of $M_{BQTRU}$, we have
\begin{eqnarray*}
(\overrightarrow{-U} , \overrightarrow{F}, -\overrightarrow{\rho(\gamma)})M_{BQTRU}&=&(\overrightarrow{-U} , \overrightarrow{F}, -\overrightarrow{\rho(\gamma)}) \left[\begin{array}{ccc}
                    qI_{4n^2} &  ~~0 &  ~~ 0 \\
                     \mathcal{\tilde{H}} &  I_{4n^2} & ~~  0\\
                     \mathfrak{D} & ~~ 0& ~  ~~ I_{4n^2}
                     \end{array}
                  \right] \\
                     &=&(-q\overrightarrow{U}+\overrightarrow{F}\mathcal{\tilde{H}}-\overrightarrow{\rho(\gamma)} \mathfrak{D}, \overrightarrow{F}, -\overrightarrow{\rho(\gamma)}).
\end{eqnarray*}
Based on Equation (\ref{product equvalence}), we have $\overrightarrow{F}\mathcal{\tilde{H}}=\overrightarrow{F\circ \tilde{H}}$ and also
\begin{eqnarray*}
\overrightarrow{\rho(\gamma)} \mathfrak{D}&=&\Big(\rho(\gamma_0) D, \dots, \rho(\gamma_3) D\Big)\\
&=&\Big(\sum_{i=1}^{n^2}\gamma_0(a_i,b_i)\overrightarrow{ \lambda_{a_i,b_i}}, \ldots , \sum_{i=1}^{n^2}\gamma_3(a_i,b_i)\overrightarrow{ \lambda_{a_i, b_i}}\Big)\\
&=&\Big(\overrightarrow{\gamma_0}, \ldots, \overrightarrow{\gamma_3}\Big)=\overrightarrow{\gamma}.
\end{eqnarray*}
We obtain  Equation (\ref{1}) resulting in $ -q U+F\circ \tilde{H} -\gamma=G$. Thus,
\begin{eqnarray*}
(\overrightarrow{-U} , \overrightarrow{F}, -\overrightarrow{\rho(\gamma)})M_{\textrm{BQTRU}}&=&(\overrightarrow{-q U+F\circ \tilde{H} -\gamma}, \overrightarrow{F}, -\overrightarrow{\rho(\gamma)})\\
  &=&(\overrightarrow{G} , \overrightarrow{F}, -\overrightarrow{\rho(\gamma)}),
\end{eqnarray*}
 and the private key is an integer linear combination of the rows of $M_{\textrm{BQTRU}}$.
 Hence, $$(\overrightarrow{G} , \overrightarrow{F}, -\overrightarrow{\rho(\gamma)})\in \Lambda_{BQTRU}.$$ $\hfill \square $
\end{proof}
The quaternion  $\gamma\in J_q$ is a linear combination of $\lambda_{a,b}$'s and therefore it may have many monomials with large coefficients. As a consequence, the Euclidean norm of  $\overrightarrow{\rho(\gamma)}$ may be very large and the private key $(\overrightarrow{G} , \overrightarrow{F}, -\overrightarrow{\rho(\gamma)}) \in \mathbb{Z}^{4n^2}\times\mathbb{Z}^{4n^2}\times\mathbb{Z}_q^{4n^2}$ may not be a short nonzero vector in the BQTRU lattice using the Euclidean norm.  Thus, we consider a hybrid metric on $\mathbb{Z}^{4n^2}\times\mathbb{Z}^{4n^2}\times\mathbb{Z}_q^{4n^2}$: Euclidean on the first two components and Hamming on the third one.
\begin{definition}
For every $(\overrightarrow{G}, \overrightarrow{F}, \overrightarrow{\rho(\gamma)})\in \mathbb{Z}^{4n^2}\times\mathbb{Z}^{4n^2}\times\mathbb{Z}_q^{4n^2}$,  the hybrid norm is defined as
  $$\|(\overrightarrow{G}, \overrightarrow{F}, \overrightarrow{\rho(\gamma)})\|_{hyb}:=\|(\overrightarrow{G}, \overrightarrow{F})\|_2+\|\overrightarrow{\rho(\gamma)}\|_H, $$
where $\|\cdot\|_2 $ and $\|\cdot\|_H $ are the Euclidean and Hamming norms, respectively.
\end{definition}
It can be easily shown that $\|\cdot\|_{hyb} $  has the properties of a norm.

\noindent In this case, since $\rho(\gamma)$ is determined by the element of the subset $T$, where $|T|$ is small, the Hamming norm of $\overrightarrow{\rho(\gamma)}$ is small. Moreover, the Euclidean norm of $\overrightarrow{F}$ and $\overrightarrow{G}$ is small. Therefore, the private key is a short vector in the lattice measured by the hybrid metric.

\textcolor{mycolor}{In the sequel, we present several results and show that  recovering the private key can be reduced to solving  an SVP in the BQTRU lattice with hybrid norm.}
\textcolor{mycolor}{\begin{lemma}[\cite{Cox}] \label{vandermonde}
Let $F$ be an arbitrary field. Given a set $U\subset F$ of size $n$, then any polynomial $f(x,y)\in F[x,y]$ of degree less than $n$ in each variable, is determined by its evaluation in $U\times U$.
\end{lemma}}
\textcolor{mycolor}{Let $U=\{b\in \mathbb{Z}_q ~|~b^n=1\}$, where the condition $n |(q-1)$ yields $\mid U\mid=n $. Then,
 Lemma \ref{vandermonde} implies that any  polynomial $f$ in $R'$ is determined by its evaluation on $U\times U=E$ or equivalently by knowing the vector $\rho(f)$.}
\textcolor{mycolor}{\begin{corollary}
Let $f$ and $g$ be two polynomials in $R'$ such that $\rho(f)=\rho(g) $. Then,  $f=g$.
\end{corollary}}
\textcolor{mycolor}{Obviously, this result is satisfied in the case that $F$ and $G$ are two quaternions in $\mathbb{A}$ where $\rho(F)=\rho(G)$.}

We use the next lemma to show that for  two fixed quaternions  $F$ and $G$ satisfying the public key equation (\ref{public key}), the vector $\overrightarrow{\rho(\gamma)}$ has the smallest Hamming norm among all quaternions that can be chosen as the third component in the private key $(\overrightarrow{G} , \overrightarrow{F}, -\overrightarrow{\rho(\gamma)})$.
\begin{lemma}\label{lemma1}
Let $F$ and $G$  be two fixed quaternions satisfying Equation (\ref{public key}). Then for any element $\tau'\neq\vartheta$ belonging to the set
$$\chi=\{\tau'\in \mathbb{A}~|~F\circ F'=F'\circ F=1 +\alpha'\circ\tau'\;(\bmod q),\;\;\tilde{H}= F'\ast G+ \tau'\;(\bmod q)\},$$
 we have $\|\overrightarrow{\rho(\tau')}\|_H > 4|T|$.
\end{lemma}
\begin{proof}
We know $\tilde{H}= F^{-1}\circ G+ \vartheta\pmod q$, where $F\circ F^{-1}=F^{-1}\circ F= 1+\alpha\circ \sigma\pmod q$. Therefore, $F\circ F^{-1}=F^{-1}\circ F= 1+\alpha''\circ W\circ\sigma\pmod q$ too, where $\alpha''\in \mathbb{A}$. Then, $\vartheta=W\circ \sigma \in \chi$.
Using Equation  (\ref{G}) and (\ref{vartheta}), we have
\begin{eqnarray*}
\rho(\tilde{H})&=&\rho(F^{-1})\bullet \rho(G)+\rho(\vartheta)\\
&=&\Big(0, \ldots, 0,  F^{-1}(a_{|T|+1}, b_{|T|+1}) \circ G(a_{|T|+1}, b_{|T|+1}), \ldots, F^{-1}(a_{n^2}, b_{n^2})\circ G(a_{n^2}, b_{n^2})\Big)\\
&& +\> \Big(q_1W, \ldots, q_{|T|}W,  0,\ldots, 0\Big),
\end{eqnarray*}
that is,
\begin{eqnarray}
\rho(\tilde{H})&=& \Big(q_1W, \ldots, q_{|T|}W,  F^{-1}(a_{|T|+1}, b_{|T|+1})\circ G(a_{|T|+1}, b_{|T|+1}),\ldots \nonumber\\
&& \> \ldots, F^{-1}(a_{n^2}, b_{n^2})\circ G(a_{n^2}, b_{n^2})\Big). \label{rho(H)}
\end{eqnarray}
If we consider an arbitrary $\tau'\in \chi$, where $\tau'\neq\vartheta$, then
$$F\circ F'=F'\circ F= 1+\alpha'\circ\tau'\;(\bmod q),\quad\tilde{H}= F'\circ G+ \tau'\;(\bmod q).$$
Similarly,
\begin{eqnarray*}
\rho(\tilde{H})&=&\rho(F')\bullet \rho(G)+\rho(\tau')\\
&=&\Big(0, \ldots, 0,  F'(a_{|T|+1}, b_{|T|+1}) \circ G(a_{|T|+1}, b_{|T|+1}), \ldots, F'(a_{n^2}, b_{n^2})\circ G(a_{n^2}, b_{n^2})\Big)\\
&& +\rho(\tau').
\end{eqnarray*}
As a result  $\tau'(a_i, b_i)=q_i W\neq 0$, for $i=1, \ldots, |T|$. On the other hand, $\tau'\neq \vartheta$ and $F'(a_i, b_i)\circ G(a_i, b_i) \neq F^{-1}(a_i, b_i)\circ G(a_i, b_i)$ for some $i=|T|+1, \ldots, n^2$. Therefore, $\rho(\tau')$ should be nonzero on at least one of the $ n^2-|T|$ last positions. Hence $\|\overrightarrow{\rho(\tau')}\|_H > 4|T|$.$\hfill \square $
\end{proof}
\begin{proposition}
Let $F$ and $G$  be two fixed quaternions satisfying Equation (\ref{public key}). Then, $\vartheta=W\circ\sigma$ has the smallest Hamming norm in $\chi$. Also, for any $\tau'\in \chi$ such that $\tau'\neq\vartheta$
 $$ \|\overrightarrow{\rho(F\circ \tau')}\|_H \geq \|\overrightarrow{\rho(F\circ \vartheta)}\|_H,$$
where, in accordance with Equation (\ref{eq3}), we have $\overrightarrow{\rho(\gamma)}=\overrightarrow{\rho(F\circ \vartheta)}$.
\end{proposition}
\begin{proof}
Since  $\vartheta\in \chi$ and $\|\overrightarrow{\rho(\vartheta)}\|_H = 4|T|$ (according to Equation (\ref{vartheta})), then  $\vartheta$ has the smallest Hamming norm in $\chi$. Let $\tau'\in \chi$,  where $\tau'\neq\vartheta$. Then, we have
$$F\circ F'=F'\circ F= 1+\alpha'\circ\tau'\;(\bmod q),\quad \tilde{H}= F'\circ G+ \tau'\;(\bmod q).$$
Therefore,
$$F\circ \tilde{H}= G+(\alpha'\circ\tau'\circ G+ F\circ\tau')\;(\bmod q).$$
Moreover, $\tilde{H}= F^{-1}\circ G+ \vartheta\pmod q$ and by considering Equation (\ref{public equation}) and (\ref{gamma}), we have
$$F\circ \tilde{H}= G+\gamma \;\;(\bmod q), $$
 where $\gamma=(\sigma\circ \alpha\circ G+ F\circ\vartheta)$. Thus,
\begin{eqnarray*}
\alpha'\circ\tau'\circ G+ F\circ\tau'&=&\sigma\circ\alpha\circ G+ F\circ\vartheta,\\
\rho(\alpha'\circ\tau'\circ G+ F\circ\tau')&=&\rho(\sigma\circ\alpha\circ G+ F\circ\vartheta),\\
\rho(\alpha'\circ\tau'\circ G)+\rho( F\circ\tau')&=&\rho(F\circ\vartheta),\\
\rho( F\circ\tau')&=&\rho(F\circ\vartheta)-\rho(\alpha'\circ\tau'\circ G).
\end{eqnarray*}
Hence, we can conclude that
$$ \|\overrightarrow{\rho(F\circ \tau')}\|_H \geq \|\overrightarrow{\rho(F\circ \vartheta)}\|_H. $$
 $\hfill \square $
 \end{proof}
An important consequence of this proposition is that for fixed quaternions $F$ and $G$  that satisfy Equation (\ref{public key}),
$\vartheta\in \chi$  is the best choice in the public key equation (\ref{public key}) and $\overrightarrow{\rho(\gamma)}=\overrightarrow{\rho(F\circ \vartheta)}$ has the smallest Hamming norm among all quaternions that can be chosen as the third component in the solution of the public key equation. Consequently,
$ \|\overrightarrow{G}, \overrightarrow{F}, \overrightarrow{\rho(F\circ \vartheta)}\|_{hyb} < \|\overrightarrow{G}, \overrightarrow{F}, \overrightarrow{\rho( F\circ \tau')}\|_{hyb}$, for any $\tau'\in \chi$.

In the following theorem which is the most important result of this section, we prove that any reduction in the Hamming norm of the third component of a solution of the public key equation (\ref{public key}) like $(\overrightarrow{G''}, \overrightarrow{F''}, \overrightarrow{\rho(F''\circ \tau'')})$, where  $\tilde{H}=F''^{-1}\circ G''+\tau''\,\,(mod\,\,q)$ and $F''\circ F''^{-1}=F''^{-1}\circ F''=1+\alpha''\circ\tau''\pmod q$, results in a considerable increase in the hybrid norm of the solution $(\overrightarrow{G''}, \overrightarrow{F''}, \overrightarrow{\rho(F''\circ \tau'')})$.
\begin{theorem}\label{original teorem}
Assume that there exist $F''$, $G'' \in \mathbb{A}$ and $\tau'' \in J_q$ such that  $\tilde{H}=F''^{-1}\circ G''+\tau''\,\,(mod\,\,q)$,  $F''\circ F''^{-1}=F''^{-1}\circ F''=1+\alpha''\circ\tau''\pmod q$, for some $\alpha''\in \mathbb{A}$ (that is, they are a solution of the public key equation) and  $\|\overrightarrow{\rho(\tau'')}\|_H< \|\overrightarrow{\rho(\vartheta)}\|_H$. Then
$$ \Big\|\overrightarrow{G''}, \overrightarrow{F''}, \overrightarrow{\rho(F''\circ \tau'')}\Big\|_{hyb}> \Big\|\overrightarrow{G}, \overrightarrow{F}, \overrightarrow{\rho( F\circ \vartheta)}\Big\|_{hyb}.$$
\end{theorem}
\begin{proof}
Let $supp(F)=\{(a, b)\in E ~|~ F(a,b)\neq 0 \}$, for any $F \in \mathbb{A}$. Using  Equation (\ref{vartheta}), we have  $|supp(\vartheta)|=|T|$.
Without loss of generality we consider that $|supp(\tau'')|=|T|-1 $ and $\rho(\tau'')$ is nonzero over the first $|T|-1 $ positions. Similar to the  initial considerations that we presented for private key and public key in our proposed scheme, if
 $F''=f''_0+ f''_1 i+f''_2 j+f''_3 k$ and $ G''=g''_0+ g''_1 i+g''_2 j+g''_3 k$ in $ \mathbb{A}$ and $\tau'' \in J_q$ are  a solution of the public key equation (\ref{public key}) for given public key $\tilde{H}$, then
\begin{eqnarray}\label{eq8}
\tilde{H}=F''^{-1}\circ G''+\tau''\;(\bmod q), \quad F''\circ F''^{-1}=F''^{-1}\circ F''=1+\alpha''\circ\tau''\;(\bmod q),
\end{eqnarray}
where,
\begin{eqnarray*}
&&supp(\tau'')=\bigcap\nolimits_{i=0}^3\{(a, b) \in E ~|~  g''_i(a, b)=0 \},\\
&&\{(a, b) \in E ~|~ N(F'')(a, b)=0 \}\subset supp(\tau'').
\end{eqnarray*}
Therefore,
\begin{eqnarray*}
\rho(\tilde{H})&=&\rho(F''^{-1})\bullet \rho(G'')+\rho(\tau'')\\
 &=&(0, \ldots, 0,  F''^{-1}(a_{|T|}, b_{|T|}) \circ G''(a_{|T|}, b_{|T|}), \ldots, F''^{-1}(a_{n^2}, b_{n^2})\circ G''(a_{n^2}, b_{n^2})\Big)\\
 && +\Big(\tau''(a_1, b_1), \ldots,\tau''(a_{|T|-1}, b_{|T|-1}), 0, \ldots, 0 \Big)\\
 &=&\Big(\tau''(a_1, b_1), \ldots,\tau''(a_{|T|-1}, b_{|T|-1}) ,  F''^{-1}(a_{|T|}, b_{|T|}) \circ G''(a_{|T|}, b_{|T|}), \ldots\\
 && \ldots, F''^{-1}(a_{n^2}, b_{n^2})\circ G''(a_{n^2}, b_{n^2})\Big).
\end{eqnarray*}
On the other hand, Equation (\ref{rho(H)}) gives
\begin{eqnarray*}
\rho(\tilde{H})&=& \Big(q_1W, \ldots, q_{|T|}W,  F^{-1}(a_{|T|+1}, b_{|T|+1})\circ G(a_{|T|+1}, b_{|T|+1}),\ldots \nonumber\\
&& \> \ldots, F^{-1}(a_{n^2}, b_{n^2})\circ G(a_{n^2}, b_{n^2})\Big).
\end{eqnarray*}
Thus, we obtain
\begin{eqnarray}
&&\tau''(a_i, b_i)=q_i W=\vartheta(a_i, b_i), \qquad i=1, \ldots, |T|-1;\label{eq7}\\
&&F''^{-1}(a_{|T|}, b_{|T|}) \circ G''(a_{|T|}, b_{|T|})=q_{|T|} W, \qquad i=|T|;\label{eq5}\\
&&F''^{-1}(a_{i}, b_{i})\circ G''(a_{i}, b_{i})=F^{-1}(a_{i}, b_{i})\circ G(a_{i}, b_{i}), \qquad i=|T|+1, \ldots, n^2. \label{eq6}
\end{eqnarray}
Therefore, $\vartheta=\tau''+(q_{|T|} W)\circ\lambda_{a_{|T|}, b_{|T|}}$ and $\tilde{H}=F^{-1}\circ G+\tau''+(q_{|T|} W)\circ\lambda_{a_{|T|}, b_{|T|}}\pmod q$. As a result
\begin{eqnarray}\label{eq4}
F''^{-1}\circ G''=F^{-1}\circ G+(q_{|T|} W)\circ\lambda_{a_{|T|}, b_{|T|}}.
\end{eqnarray}
We can consider
\begin{eqnarray*}
F''^{-1}&=& F^{-1}+\sum\nolimits_{i= 1}^{n^2} \Big(F''^{-1}(a_{i}, b_{i})-F^{-1}(a_{i}, b_{i})\Big)\circ\lambda_{a_{i}, b_{i}},\\
G''&=&G+\sum\nolimits_{i= 1}^{n^2} \Big(G''(a_{i}, b_{i})-G(a_{i}, b_{i})\Big)\circ\lambda_{a_{i}, b_{i}},
\end{eqnarray*}
since by  Lemma \ref{vandermonde} they are equal over the set $E$. Next,
\begin{eqnarray}
F''^{-1}\circ G''&=&F^{-1}\circ G+ \sum\nolimits_{i= 1}^{n^2} \Big(F''^{-1}(a_{i}, b_{i})\circ G''(a_{i}, b_{i})-F^{-1}(a_{i}, b_{i})\circ G(a_{i}, b_{i})\Big)\circ\lambda_{a_{i}, b_{i}} \nonumber\\
&=&F^{-1}\circ G+ \sum\nolimits_{i= |T|+1}^{n^2} \Big(F''^{-1}(a_{i}, b_{i})\circ G''(a_{i}, b_{i})-F^{-1}(a_{i}, b_{i})\circ G(a_{i}, b_{i})\Big)\circ\lambda_{a_{i}, b_{i}}\nonumber\\
&&+\Big(F''^{-1}(a_{|T|}, b_{|T|})\circ G''(a_{|T|}, b_{|T|})\Big)\circ\lambda_{a_{|T|}, b_{|T|}}\nonumber\\
&=&F^{-1}\circ G+ \sum\nolimits_{i= |T|+1}^{n^2} \Big(F''^{-1}(a_{i}, b_{i})\circ G''(a_{i}, b_{i})-F^{-1}(a_{i}, b_{i})\circ G(a_{i}, b_{i})\Big)\circ\lambda_{a_{i}, b_{i}}\nonumber\\
&&+\Big(q_{|T|} W\Big)\circ\lambda_{a_{|T|}, b_{|T|}}.\nonumber
\end{eqnarray}
From the last equation and Equation (\ref{eq4}) we have
$$\sum\nolimits_{i= |T|+1}^{n^2} \Big(F''^{-1}(a_{i}, b_{i})\circ G''(a_{i}, b_{i})-F^{-1}(a_{i}, b_{i})\circ G(a_{i}, b_{i})\Big)\circ\lambda_{a_{i}, b_{i}}=0,$$
that is, $F''^{-1}(a_{i}, b_{i})$ and $G''(a_{i}, b_{i})$ satisfy Equation (\ref{eq6}). Without loss of generality, we take $F''^{-1}(a_{i}, b_{i})=F^{-1}(a_{i}, b_{i})$ and $G''(a_{i}, b_{i})= G(a_{i}, b_{i})$, for $i=|T|+1, \ldots, n^2 $, then
\begin{eqnarray}
F''^{-1}&=& F^{-1}+\sum\nolimits_{i= 1}^{|T|} \big(F''^{-1}(a_{i}, b_{i})-F^{-1}(a_{i}, b_{i})\big)\circ\lambda_{a_{i}, b_{i}},\\
G''&=&G+\sum\nolimits_{i= 1}^{|T|} \big(G''(a_{i}, b_{i})-G(a_{i}, b_{i})\big)\circ\lambda_{a_{i}, b_{i}}.\label{G''}
\end{eqnarray}
Next, we prove that
\begin{eqnarray}\label{F}
F''= F+\sum\nolimits_{i= 1}^{|T|} \big(F''(a_{i}, b_{i})-F(a_{i}, b_{i})\big)\circ\lambda_{a_{i}, b_{i}}.\label{F''}
\end{eqnarray}
Indeed, we show that $F''$ as above satisfies the equation $F''\circ F''^{-1}=F''^{-1}\circ F''=1+\theta\circ\tau''\pmod q$, where $\theta \in\mathbb{A}$. We know  $F\circ F^{-1}=F^{-1}\circ F= 1+\beta\circ\vartheta\pmod q$, for some $\beta\in \mathbb{A}$. Therefore, in accordance to Equation (\ref{eq8}) and (\ref{eq7}), we have
\begin{eqnarray*}
\lefteqn{F''\circ F''^{-1}}\\
&&= F\circ F^{-1}+\sum\nolimits_{i= 1}^{|T|} \Big(F''(a_{i}, b_{i})\circ F''^{-1}(a_{i}, b_{i})-F(a_{i}, b_{i})\circ F^{-1}(a_{i}, b_{i})\Big)\circ\lambda_{a_{i}, b_{i}}\\
&&= 1+ \beta\circ\vartheta+  \Big(F''(a_{|T|}, b_{|T|})\circ F''^{-1}(a_{|T|}, b_{|T|})-F(a_{|T|}, b_{|T|})\circ F^{-1}(a_{|T|}, b_{|T|})\Big)\circ\lambda_{a_{|T|}, b_{|T|}}\\
&&\quad  + \sum\nolimits_{i= 1}^{|T|-1}  \Big(F''(a_{i}, b_{i})\circ F''^{-1}(a_{i}, b_{i})-F(a_{i}, b_{i})\circ F^{-1}(a_{i}, b_{i})\Big)\circ\lambda_{a_{i}, b_{i}}\\
&&= 1+ \beta\circ\vartheta+ \Big(1-\big(1+\beta(a_{|T|}, b_{|T|})\circ\vartheta(a_{|T|}, b_{|T|})\big)\circ\lambda_{a_{|T|}, b_{|T|}} \Big)\\
&&\quad  + \sum\nolimits_{i= 1}^{|T|-1} \Big(\big(1+\alpha''(a_{i}, b_{i})\circ \tau''(a_{i}, b_{i})\big)-\big(1+ \beta(a_{i}, b_{i})\circ \vartheta(a_{i}, b_{i})\big)\Big)\circ\lambda_{a_{i}, b_{i}}\\
&&= 1+ \beta\circ\vartheta-\beta(a_{|T|}, b_{|T|})\circ\vartheta(a_{|T|}, b_{|T|})\circ\lambda_{a_{|T|}, b_{|T|}}
 + \sum\nolimits_{i= 1}^{|T|-1} \xi(a_{i}, b_{i})\circ \tau''(a_{i}, b_{i})\circ\lambda_{a_{i}, b_{i}}\\
&&= 1+ \sum\nolimits_{i= 1}^{|T|} \beta(a_{i}, b_{i})\circ\vartheta(a_{i}, b_{i})\circ\lambda_{a_{i}, b_{i}}-\beta(a_{|T|}, b_{|T|})\circ\vartheta(a_{|T|}, b_{|T|})\circ\lambda_{a_{|T|}, b_{|T|}}\\
&& \quad + \sum\nolimits_{i= 1}^{|T|-1} \xi(a_{i}, b_{i})\circ \tau''(a_{i}, b_{i})\circ\lambda_{a_{i}, b_{i}}\\
&&= 1+ \sum\nolimits_{i= 1}^{|T|-1} \beta(a_{i}, b_{i})\circ\vartheta(a_{i}, b_{i})\circ\lambda_{a_{i}, b_{i}}+ \sum\nolimits_{i= 1}^{|T|-1} \xi(a_{i}, b_{i})\circ \tau''(a_{i}, b_{i})\circ\lambda_{a_{i}, b_{i}}\\
&&= 1+ \sum\nolimits_{i= 1}^{|T|-1} \beta(a_{i}, b_{i})\circ\tau''(a_{i}, b_{i})\circ\lambda_{a_{i}, b_{i}}+ \sum\nolimits_{i= 1}^{|T|-1} \xi(a_{i}, b_{i})\circ \tau''(a_{i}, b_{i})\circ\lambda_{a_{i}, b_{i}}\\
&&= 1+ \tau''\circ \Big(\sum\nolimits_{i= 1}^{|T|-1} \beta(a_{i}, b_{i})\circ\lambda_{a_{i}, b_{i}}\Big)+ \tau''\circ \Big(\sum\nolimits_{i= 1}^{|T|-1} \xi(a_{i}, b_{i})\circ\lambda_{a_{i}, b_{i}}\Big)\\
&&= 1+\beta \circ \tau''+\xi \circ \tau''=1+\theta \circ \tau''.
\end{eqnarray*}
Thus,
$F''= F+\sum\nolimits_{i= 1}^{|T|} \big(F''(a_{i}, b_{i})-F(a_{i}, b_{i})\big)\circ\lambda_{a_{i}, b_{i}}$ and
\begin{eqnarray*}
\|\overrightarrow{F''}\|_2^2&=&\|\overrightarrow{F}\|_2^2+\Big\|\sum\nolimits_{i= 1}^{|T|} \overrightarrow{\big(F''(a_{i}, b_{i})-F(a_{i}, b_{i})\big)\circ\lambda_{a_{i}, b_{i}}}\Big\|_2^2\\
&&+2\Big(\overrightarrow{F}\cdot \big(\sum\nolimits_{i= 1}^{|T|} \overrightarrow{\big(F''(a_{i}, b_{i})-F(a_{i}, b_{i})\big)\circ\lambda_{a_{i}, b_{i}}(x, y)\big)}\Big),
\end{eqnarray*}
where $`` \cdot "$ is the inner product.
\textcolor{mycolor}{Since $F$ is a small ternary quaternion polynomial, $\overrightarrow{F}$  is a vector with most elements equal to zero and the rest are $1$ or $-1$. As a consequence,   many products in  $\Big(\overrightarrow{F}\cdot \big(\sum\nolimits_{i= 1}^{|T|} \overrightarrow{\big(F''(a_{i}, b_{i})-F(a_{i}, b_{i})\big)\circ\lambda_{a_{i}, b_{i}}\big)}\Big)$ vanish. On the other hand, the polynomials $\lambda_{a_{i}, b_{i}}$, for $i=0, \ldots, |T|$, have many monomials with large coefficients and then the vectors $\overrightarrow{\lambda_{a_{i}, b_{i}}}$'s have many nonzero coefficients with large values belonging to $\mathbb{Z}_q$.  Therefore, the Euclidian norm of $\sum\nolimits_{i= 1}^{|T|} \overrightarrow{\big(F''(a_{i}, b_{i})-F(a_{i}, b_{i})\big)\circ\lambda_{a_{i}, b_{i}}}$ is very larger with respect to  $\Big(\overrightarrow{F}\cdot \big(\sum\nolimits_{i= 1}^{|T|} \overrightarrow{\big(F''(a_{i}, b_{i})-F(a_{i}, b_{i})\big)\circ\lambda_{a_{i}, b_{i}}\big)}\Big)$.} Hence,
$$\Big\|\sum\nolimits_{i= 1}^{|T|} \overrightarrow{\big(F''(a_{i}, b_{i})-F(a_{i}, b_{i})\big)\circ\lambda_{a_{i}, b_{i}}}\Big\|_2^2  > 2\Big(\overrightarrow{F}\cdot \big(\sum\nolimits_{i= 1}^{|T|} \overrightarrow{\big(F''(a_{i}, b_{i})-F(a_{i}, b_{i})\big)\circ\lambda_{a_{i}, b_{i}}\big)}\Big),$$
and the immediate result is
$$\|\overrightarrow{F''}\|_2 > \|\overrightarrow{F}\|_2.$$
Under the same computations and according to Equation (\ref{G''}), we have $\|\overrightarrow{G''}\|_2 > \|\overrightarrow{G}\|_2.$

Putting all pieces together,  a small reduction in the Hamming norm of $\overrightarrow{\rho(\tau'')}$, leads to a small reduction in the Hamming norm of $\overrightarrow{\rho(F''\circ \tau'')}$,  \textcolor{mycolor}{entailing a high  Euclidean norm for $\overrightarrow{F''}$ and $\overrightarrow{G''}$. This gives a considerable  enlargement for the hybrid norm of $(\overrightarrow{G''}, \overrightarrow{F''}, \overrightarrow{\rho(F''\circ \tau'')})$.} One concludes that  the private key  $(\overrightarrow{G}, \overrightarrow{F}, \overrightarrow{\rho(\gamma)})=(\overrightarrow{G}, \overrightarrow{F}, \overrightarrow{\rho( F\circ \vartheta)}) $ is mostly likely one of the nonzero shortest vectors in $\Lambda_{BQTRU}$.
 $\hfill \square $
\end{proof}
\textcolor{mycolor}{We remark that a consequence of the previous theorem is that recovering the private key in our cryptosystem can be reduced to solving an SVP in a lattice with hybrid norm.}
\subsection{Lattice based attack}
The lattice based attacks, presented by Coppersmith and Shamir \cite{Coppersmith}, are the most important attacks on the NTRU cryptosystem. The purposes of this type of attacks can be classified in two subclasses: an attempt at finding the private key and an attempt at obtaining the message. In this section, we present a lattice based attack  that attempts at obtaining the private key. We show that our cryptosystem is secure against lattice based attacks.

In Theorem \ref{original teorem} it is shown that the private key $(\overrightarrow{G}, \overrightarrow{F}, \overrightarrow{\rho(\gamma)}) $ is with high probability one of the nonzero shortest vectors, measured by the hybrid metric, in $\Lambda_{BQTRU}$. Solving SVP or CVP with the hybrid metric is a challenging problem and as far as we know no algorithm has been developed to handle them.
In the following we expand our BQTRU lattice to a lattice for which the SVP and CVP are measured using the Euclidean norm.

We select a small subset $\{c_0,c_1,\ldots, c_l\}$ of elements of $\mathbb{Z}_q$  such that each element $z \in \mathbb{Z}_q$ can be represented as $\sum\nolimits_{i = 0}^l  z_i\cdot c_i$, where $(z_0, z_1,\ldots, z_l)$  is a vector with small Euclidean norm. In this work, we consider $\{c_0,c_1,\ldots, c_l\}=\{1, 2,\ldots, 2^{l}\}$, where $ 2^{l}\leq q < 2^{l+1}$. Then, we expand the BQTRU lattice to the lattice which is generated by the rows of the following matrix
\begin{eqnarray*}
M_{\textrm{expanded}}=\left[\begin{array}{ccc}
                    qI_{4n^2} &  0 &   0 \\
                     \mathcal{\tilde{H}} &  I_{4n^2} &  0\\
                     \mathcal{D} &  0& I_{4(l+1)n^2}
                     \end{array}
                  \right],
\end{eqnarray*}
where
\begin{eqnarray*}
 \mathcal{D} =\left[\begin{array}{cccc}
                      \mathbf{c}\otimes D  & 0    & 0   & 0\\
                      0   & \mathbf{c}\otimes D  & 0   & 0\\
                      0 & 0  & \mathbf{c}\otimes D   & 0\\
                      0   & 0   & 0  & \mathbf{c}\otimes D
                       \end{array}\right], \qquad \mathbf{c}=\left[\begin{array}{cccc} c_0, c_1, \ldots, c_l\end{array}\right]^{Tr}.
\end{eqnarray*}
Then, we solve SVP in this lattice, $\Lambda_{\textrm{expanded}}$, that is measured under the Euclidean metric.
Finally, the points of the expanded lattice are mapped to the points of the original lattice as follows
\begin{eqnarray*}
\psi: \mathbb{Z}^{4n^{2}}\times  \mathbb{Z}^{4n^{2}}\times \overbrace{\mathbb{Z}_2^{n^{2}} \times \cdots\times\mathbb{Z}_2^{n^{2}}}^{4(l+1) ~\textrm{times}} & \rightarrow & \mathbb{Z}^{4n^{2}}\times  \mathbb{Z}^{4n^{2}}\times \overbrace{\mathbb{Z}_q^{n^{2}} \times\cdots\times \mathbb{Z}_q^{n^{2}}}^{4 ~\textrm{times}}\\
(r, s, B_{0,0}, \ldots, B_{0,l}, B_{1,0}, \ldots, B_{1,l}, \ldots, B_{3,0}, \ldots, B_{3,l}) &\mapsto & (r, s, \sum\limits_{i = 0}^l B_{0,i} c_i, \sum\limits_{i = 0}^l B_{1,i} c_i, \ldots, \sum\limits_{i = 0}^l B_{3,i} c_i).
\end{eqnarray*}
The binary representation  of any $z \in \mathbb{Z}_q$ is unique and all $z_i$ in the linear combination $z=\sum\nolimits_{i = 0}^l  z_i\cdot c_i$ are one or zero. Therefore, the Euclidean norm of $(z_0, z_1,\ldots, z_l)$ is small.  Hence, if we consider $\{c_0,c_1,\ldots, c_l\}=\{1, 2,\ldots, 2^{l}\}$, where $l=\lfloor \log_2 q \rfloor$, then any $z \in \mathbb{Z}_q$ can be written as a linear combination of them.  Since the binary representation of any $z \in \mathbb{Z}_q$ is unique, then the mapping $\psi$ is one to one and this mapping returns a unique solution.

We show that the equivalent vector for the private key $(\overrightarrow{G}, \overrightarrow{F}, \overrightarrow{\rho(\gamma)}) $ in the expanded lattice is its shortest vector under the Euclidean norm.
If we consider  the binary representation  of any point in $\rho(\gamma_i)$, for $i=0, \ldots, 3$, we have
\begin{eqnarray*}
\rho(\gamma_i)&=& \big(\gamma_i(a_1, b_1), \gamma_i(a_2, b_2), \ldots, \gamma_i(a_{n^2}, b_{n^2})\big)\\
&=& \Big(\sum\limits_{j = 0}^l {z_{i,1j} c_j}, \sum\limits_{j = 0}^l z_{i, 2j} c_j, \ldots, \sum\limits_{j= 0}^l {z_{i, n^{2}j} c_j}\Big).
\end{eqnarray*}
Also
\begin{eqnarray*}
\rho(\gamma_i) D&=&\Big(\sum\limits_{j = 0}^l {z_{i,1j} c_j}, \sum\limits_{j = 0}^l z_{i, 2j} c_j, \ldots, \sum\limits_{j= 0}^l {z_{i, n^{2}j} c_j}\Big) D\\
&=&\Big(\sum\limits_{j= 0}^l ( {z_{i,1j}},{z_{i,2j}}, \ldots ,{z_{i,{n^2}j}})\Big){c_j}D\\
&=&\big((z_{i,10}, z_{i,20}, \ldots, z_{i,{n^2}0}), (z_{i,11}, z_{i,21}, \ldots, z_{i,{n^2}1}), \ldots, (z_{i,1l}, z_{i,2l}, \ldots, z_{i,{n^2}l})\big) \left(\begin{array}{c}
c_0 D\\
c_1 D \\
\vdots\\
c_l D
 \end{array}
 \right)
\\
&=& (A_{i,0}, A_{i,1}, \ldots, A_{i,l}) \left(\begin{array}{c}
c_0 D\\
c_1 D \\
\vdots\\
c_l D
 \end{array}
 \right)\\
 &=&\sum\limits_{j= 0}^l {A_{i,j} c_j D}.
\end{eqnarray*}
Therefore,  $(\overrightarrow{G}, \overrightarrow{F}, A_{0,0}, \ldots, A_{0,l}, \ldots, A_{3,0}, \ldots, A_{3,l}) $ is the equivalent vector of the private key and
\begin{eqnarray*}
\lefteqn{\Big(\overrightarrow{-U} , \overrightarrow{F},A_{0,0}, \ldots, A_{0,l}, \ldots, A_{3,0}, \ldots, A_{3,l}\Big)M_{expand}}\\
&& = \Big(-q \overrightarrow{U}+\overrightarrow{F}\mathcal{\tilde{H}}+\big(\sum\limits_{i = 0}^l {A_{0,i} c_i D}, \ldots , \sum\limits_{i = 0}^l {A_{3,i} c_i D}\big), \overrightarrow{F}, A_{0,0}, \ldots, A_{0,l}, \ldots, A_{3,0}, \ldots, A_{3,l}\Big)\\
&& = \Big(-q \overrightarrow{U}+\overrightarrow{F}\mathcal{\tilde{H}}+\big(\rho(\gamma_0) D, \ldots, \rho(\gamma_3) D\big), \overrightarrow{F}, A_{0,0}, \ldots, A_{0,l}, \ldots, A_{3,0}, \ldots, A_{3,l}\Big)\\
&& = \Big(-q\overrightarrow{U}+\overrightarrow{F}\mathcal{\tilde{H}}+\overrightarrow{\rho(\gamma)} \mathfrak{D}, \overrightarrow{F},A_{0,0}, \ldots, A_{0,l},\ldots, A_{3,0}, \ldots, A_{3,l}\Big)\\
&&  =\Big(\overrightarrow{-q U+F\circ \tilde{H} +\gamma}, \overrightarrow{F},A_{0,0}, \ldots, A_{0,l}, \ldots, A_{3,0}, \ldots, A_{3,l}\Big)\\
&&  =\Big(\overrightarrow{G} , \overrightarrow{F}, A^0_0, \ldots, A_{0,0}, \ldots, A_{0,l}, \ldots, A_{3,0}, \ldots, A_{3,l}\Big).
\end{eqnarray*}
We conclude that, $(\overrightarrow{G}, \overrightarrow{F}, A_{0,0}, \ldots, A_{0,l}, \ldots, A_{3,0}, \ldots, A_{3,l}) \in \Lambda_{\textrm{expanded}}$. Next we show that it is one of the shortest vector in the expanded lattice.

In the following proposition, we consider the recommended parameters in Section \ref{Parameter Selection}  in such a way that the possibility of decryption failure be negligible. Choosing the prime number $q$ such that $n |(q-1)$ and  $q>24dp$ in the BQTRU cryptosystem ensures that the decryption never fails. Following the method of selecting parameters in NTRU (see Proposition $6.61$ of \cite{book}), for simplicity, we can remove $p$ in the last equation and assume  that $q\approx 24d$. Moreover,  $q=\alpha n+1$, for some $\alpha \in \mathbb{N}$. Hence, since $d\approx n^2/7$ in BQTRU, we consider a prime number $q$, where $q\approx 24n^2/7+1$ in the following proposition.
\begin{proposition}
Let $(n,p,q,d)$ be BQTRU parameters such that
$$d\approx n^2/7, \quad q\approx 24n^2/7+1, \quad |T|\leq n.$$
Then the Gaussian heuristic predicts that the equivalent vector of the private key in $\Lambda_{\textrm{expanded}}$, that is $(\overrightarrow{G} , \overrightarrow{F}, A_{0,0}, \ldots, A_{0,l}, \ldots, A_{3,0}, \ldots, A_{3,l})$, is the shortest nonzero vector in the $\Lambda_{expand}$ with maximum likelihood.
\end{proposition}
\begin{proof}
The average length of the shortest nonzero vector in a lattice $\Lambda$ with dimension $n$  is  predicted by the Gaussian heuristic \cite{book} (see page 377) as
$$\lambda_1(\Lambda)=\sqrt{\frac{n}{2\pi e}}(\det(\Lambda))^{\frac{1}{n}}.$$
This prediction in the extended lattice with dimension $(4l+12)n^2$ and $\det(\Lambda_{\textrm{expand}})=q^{4n^2}$ is
\begin{eqnarray*}
\lambda_1(\Lambda_{\textrm{expand}})&=&\sqrt{\frac{(4l+12)n^2}{2\pi e}}q^{\frac{4}{(4l+12)}}.
\end{eqnarray*}
Since $ 2^{l}\leq q < 2^{l+1}$, we have
$$\lambda_1(\Lambda_{\textrm{expand}})\geq \sqrt{\frac{(2l+6)n^2}{\pi e}}2^{\frac{l}{(l+3)}},$$
where $l=\lfloor \log_2 q \rfloor$ and we assume $q\approx 24n^2/7+1\approx3.4n^2+1$. Therefore, we can paste $l\approx \log_2 (3.4n^2+1)$ and have
\begin{eqnarray*}
\lambda_1(\Lambda_{\textrm{expand}})\geq\sqrt{\frac{\big((2\log_2(3.4n^2+1)+6)n^2\big)\big((3.4n^2+1)^2\big)^{1/(\log_2(27.4n^2+8))}}{\pi e}}.
\end{eqnarray*}
On the other hand, if $\|\overrightarrow{\rho(\gamma)}\|_H =t'$, then
$$\|g , f, A_{0,0}, \ldots, A_{0,l}, \ldots, A_{3,0}, \ldots, A_{3,l}\|_2\leq \sqrt{16d+ t'(l+1)}.$$
Since $d\approx n^2/7$ and  $t'< 4|T|$, then
$$\|g , f, A_{0,0}, \ldots, A_{0,l}, \ldots, A_{3,0}, \ldots, A_{3,l}\|_2 \leq \sqrt{\frac{16n^2}{7}+ 4|T|(\log_2(3.4n^2+1)+1)}.$$
In general,  we consider $|T|\leq n$ and hence for $n\geq 7$,
 \begin{eqnarray*}
 && \sqrt{\frac{16n^2}{7}+ 4n(\log_2(3.4n^2+1)+1)}\\
 \leq && \sqrt{\frac{\big((2\log_2(3.4n^2+1)+6)n^2\big)\big((3.4n^2+1)^2\big)^{1/(\log_2(27.4n^2+8))}}{\pi e}}.
 \end{eqnarray*}
This proves that
$$\|g , f, A_{0,0}, \ldots, A_{0,l}, \ldots, A_{3,0}, \ldots, A_{3,l}\|_2 \leq \lambda_1(\Lambda_{\textrm{expand}}),$$
and if $n$ is  large, with high probability, $(\overrightarrow{G} , \overrightarrow{F}, A_{0,0}, \ldots, A_{0,l}, \ldots, A_{3,0}, \ldots, A_{3,l})$ is the shortest nonzero vector in the expanded lattice using Euclidean norm.
 $\hfill \square $
\end{proof}

The image of the vector in the proof of the previous proposition is the private key $(\overrightarrow{G} , \overrightarrow{F}, \overrightarrow{\rho(\gamma)})$  under the mapping $\psi$. Indeed, finding the approximate solution of SVP in the hybrid metric leads to expanding this lattice to a lattice with a higher dimension and finding SVP in the expanded lattice. Therefore, finding the private key in our cryptosystem with parameters $(n, p, q)$ is equivalent to finding a shortest vector in the expanded lattice with dimension $(4l+12)n^2=(4\log_2(3.4n^2+1)+12)n^2$.

In this case, since  $\dim(\Lambda_{\textrm{NTRU}})=8n^2$ and  $\dim(\Lambda_{\textrm{expand}})=12n^2+4n^2\log_2(3.4n^2+1)$, the dimension of $\Lambda_{\textrm{expand}}$ is at least  $(2+\log_2 n)$ times larger than  the dimension of NTRU lattice with parameters $(4n^2, q', p)$:
\begin{eqnarray*}
\frac{\dim(\Lambda_{\textrm{expand}})}{\dim(\Lambda_{\textrm{NTRU}})}&=&\frac{12n^2+4n^2\log_2(3.4n^2+1)}{8n^2}
=\frac{3}{2}+\frac{\log_2(3.4n^2+1)}{2}\\
&\geq&\frac{3}{2}+\frac{\log_2 3.4}{2}+\frac{\log_2 n^2}{2}\geq \frac{3}{2}+\frac{1.7}{2}+\log_2 n\geq 2+\log_2 n.
\end{eqnarray*}
For example when $n=7$, the dimension of the expanded lattice is approximately $2036$ while the dimension of $(4n^2, q', p)$-NTRU lattice is $392$. This increase in the dimension of the lattice forces high cost  to the attacker of the cryptosystem. We note that only the attacker,  for key recovery, faces with this higher dimensions caused  by the expansion. Since the computational complexity increases as a result of this expansion, this cryptosystem achieves a good security against lattice attacks.

We observe that we can reduce the size of the private keys $F$ and $G$ without violating the system security. This allows the use of larger messages. As a consequence, the security of the message is also increased.
\subsection{Brute force attack}
In our cryptosystem, the public parameters are $n$,  $ p$, $q$, $d_f$, $d_g$, $d_\phi$ and the public key is equal to  $ \tilde{H}=F^{-1} \circ G \pmod{J}$, where $J=\langle\sigma, q\rangle_{\mathbb{A}}$. For a brute-force search, Eve searches  exhaustively every possible key and tests whether it is the correct one or not. In order to break our cryptosystem, she searches all quaternions $F$, $G$ and polynomials $\sigma$ which satisfy the equation $F\circ\tilde{H}= G \pmod{\langle\sigma, q\rangle_{\mathbb{A}}}$.
Therefore, she may first enumerate all candidates for the ideal $Q=\langle\sigma, q\rangle$ that can be used for decryption. Then she tries each quaternion $F=f_0+f_1 i+f_2 j+f_3 k \in \mathbb{A}$ where $f_i\in L_f$, for $i=0, \ldots, 3$, until she finds a decryption key. That is, she finds the private key $F$ by verifying that $F\circ \tilde{H} \pmod J$ is a small ternary quaternion polynomial. Eve could also  find the private key by trying all possible quaternions   $G=g_0+g_1 i+g_2 j+g_3 k \in \mathbb{A}$ where $g_i\in L_g$, for $i=0, \ldots, 3$,  and $ G\circ H^{-1} \pmod J$ is a small  ternary quaternion.

The number of possible ideals $Q$ generated by  $\sigma(x,y)=\sum\nolimits_{i=1}^{|T|} q_i\lambda_{a_i, b_i}(x,y)$ is equal to
 $$\sum\limits_{i= 1}^n (q-1)^{i}{n^2\choose i}$$
and the number of possible small  ternary quaternions $F$ and $G$ is
\begin{eqnarray*}
\# L_F&=& {n^2\choose d_f}^4 {n^2-d_f\choose d_f}^4,\\
\# L_G&=&{n^2\choose d_g}^4 {n^2-d_g\choose d_g}^4.
\end{eqnarray*}
 On the other hand, $\rho(\gamma)=\rho(F\circ \vartheta)$ where $\vartheta=W\circ\sigma $ and $W \in \mathbb{A}_q^{\star}$. Therefore,  she needs to check $|\mathbb{A}_q^{\star}|$ quaternion $W$ to find $\rho(\gamma)$. Since $d_g$ is generally smaller than $d_f$, $L_G$ is smaller than $L_F$.
Hence, using a Meet-In-The-Middle  attack approach \cite{MITM}, the search space for recovering the private key $(\overrightarrow{G}, \overrightarrow{F}, -\overrightarrow{\rho(\gamma)}) $ is approximately
\begin{eqnarray*}
 {\textrm{Key} \choose \textrm{Security}}&=&|\mathbb{A}_q^{\star}| {n^2\choose d_g}^2 {n^2-d_g\choose d_g}^2\sum\limits_{i= 1}^n (q-1)^{i}{n^2\choose i},
\end{eqnarray*}
where $\mathbb{A}_q^{\star}$ is the set of invertible elements in $\mathbb{A}_q$.

Using another point of view, Eve can just try to recover the original message $M$ from the corresponding ciphertext $C$, instead of making a total break by finding the private key  $(\overrightarrow{G}, \overrightarrow{F}, -\overrightarrow{\rho(\gamma)}) $.
She simply tries  all  random quaternion $\Phi=\phi_0+\phi_1 i+\phi_2 j+\phi_3 k \in \mathbb{A} $ such that $\phi_i \in L_\phi$ for $i=0,1,2,3$ and  test  whether $C-p\tilde{H}\circ \Phi\pmod q$ is a small ternary quaternion or not.
If the obtained quaternion has small coefficients it might be the original  message $M$. Hence the search space for recovering the message (using a Meet-In-The-Middle  attack approach) is
$$ {\textrm{Message} \choose \textrm{Security}}= {n^2\choose d_\phi}^2 {n^2-d_\phi\choose d_\phi}^2= \dfrac{(n^2!)^2}{(d_\phi!)^4 \, (n^2-2d_\phi)!^2}.$$
\begin{table}[ht]
\caption{\textcolor{mycolor2}{Public key size, }parameter sets and their specified security levels with comparisons to  NTRU \cite{NTRU} and QTRU \cite{QTRU} cryptosystems (in each case $p=3$).}
\centering
\tiny
\renewcommand{\arraystretch}{1.5}
\begin{tabular}{|c||c|c|c|c|c|c|c|c|c|c|}
 \hhline{-||----------}
Sec. level & Cryptosystem & $n$ & $q$ & $d_f$ & $d_g$  & $d_\phi$ & Key Sec. & Msg. Sec. & Successful dec. & \textcolor{mycolor2}{Public key size (bit)} \\
  \hhline{=::==========}
Moderate & NTRU & $167$ & $128$  & $60$ & $20$ & $18$ & $2^{82}$ & $2^{77}$ & $0.9999292925$ & $1169$\\
  \hhline{-||----------}
Highest & NTRU & $503$ & $256$ & $215$ & $72$ & $55$ & $2^{285}$ & $2^{170}$ & $0.9999532447$ & $4024$ \\
 \hhline{-||----------}
Moderate & QTRU & $47$ & $59$ & $7$ & $6$ & $5$ & $2^{90}$ & $2^{80}$ & $0.9986364660$ &  $1128$ \\
  \hhline{-||----------}
Highest & QTRU & $149$ & $191$ & $22$ & $15$ & $12$ & $2^{263}$ & $2^{225}$ & $0.9999998808$ & $4768$ \\
  \hhline{-||----------}
Moderate & BQTRU & $7$ & $113$ & $7$ & $6$ & $6$ & $>2^{166}$ & $2^{92}$ & $0.9985784846$ & $1372$ \\
  \hhline{-||----------}
Highest & BQTRU & $11$ & $199$ & $17$ & $17$ & $13$ & $>2^{396}$ & $2^{212}$ & $0.9999995349$ & $3872$ \\
  \hhline{-||----------}
 \end{tabular}
 \label{table1}
 \end{table}
 The key and message security estimates for the presented parameter sets in  NTRU, QTRU and BQTRU cryptosystems are given in Table \ref{table1}. The provided lower bounds on the key security for BQTRU are given by assuming $|\mathbb{A}_q^{\star}|=1$.
According to this table, the brute-force attack on BQTRU with values greater than $n=11$ appears to be practically impossible.
\textcolor{mycolor2}{That is a necessary condition but not really a strong argument for security.}

\textcolor{mycolor2}{Moreover, the public key size of these cryptosystems is presented in Table \ref{table1}.  The public key in NTRU cryptosystem is a polynomial in $R_{q}=\mathbb{Z}_q[x]/\langle x^n -1 \rangle$, therefore its key size is equal to  $n\lceil\log_2 q\rceil$ bits. In QTRU cryptosystem, the public key includes four polynomials  in $R_{q}$ and then the public key size is equal to  $4n\lceil\log_2 q\rceil$ bits. Furthermore, the public key size of the proposed scheme is  $4n^2\lceil\log_2 q\rceil$ bits due to the four public polynomials $ h_i\in R'_{q}=\mathbb{Z}_q[x, y]/\langle x^n-1, y^n-1\rangle$, for $i=0, \ldots, 3$. According to the result of Table \ref{table1}, the public key size of BQTRU is smaller than NTRU and QTRU with higher security level.}

See Section \ref{Parameter Selection} for more details about the parameter's selection. These parameter sets,  which are obtained from our experiments, are chosen in such a way that the parameter $q$ be a prime number, $n |(q-1)$ and $q\approx\frac{3}{7}q'$, where $q'$ is equivalent parameter in $(4n^2, q', p)$-NTRU. Then, $(n, q, p)$-BQTRU has an equivalent level of security  as $(4n^2, q', p)$-NTRU.
We compare the combinatorial security of NTRU, QTRU and BQTRU in light of the meet-in-the-middle attack, with the prescribed parameter sets for two different security levels.

\subsection{The Gentry attack}
The parameter $n$ of the ring $R=\mathbb{Z}[x]/<x^n-1>$ in the original version of  NTRU cryptosystem, is chosen to be a prime number, because having $n$ prime maximizes the probability that the private key has an inverse with respect to a specified modulus. In 1999, Silverman has proposed taking $n$ to be a power of $2$ to allow the use of Fast Fourier Transforms when computing the convolution product of elements in the ring \cite{Silverman}. In the Gentry's attack \cite{Message Attack}, it is shown that choosing a composite number $n=cd$, especially one  with a small factor, significantly reduces the security of the NTRU cryptosystem.
%
Indeed, when $n=cd$ is a composite number,  a ring homomorphism  is used  to construct much smaller (and more easily reduceable) lattices whose shortest vectors contain at least some useful cryptanalytic information about the secret key.

In general, when $n$ is composite and $d$ is a nontrivial divisor, there exist a ring homomorphism
\begin{eqnarray*}
\mathbb{Z}[x]/(x^n-1)&\longrightarrow&  \mathbb{Z}[x]/(x^d- 1) \\
 f &\mapsto& f(d).
 \end{eqnarray*}
Then a $2d$-dimensional lattice $\Lambda(d)$ analog to the NTRU lattice can be  constructed. This lattice contains the vector $(f(d), g(d))$.
Therefore, at least when $d >\sqrt{n}$, $(f(d), g(d))$ is almost certainly the shortest vector in $\Lambda(d)$ for the same reasons that $(f, g)$ is almost certainly the shortest vector for the NTRU lattice.
Thus, significant partial information about the private key is recovered by reducing to a $ 2d$-dimensional lattice.
Alternatively, we can use $(f(d), g(d))$ to recover $(f, g)$.
We thereby obtain the private key without ever having to reduce a $2n$-dimensional
lattice in NTRU. Indeed, we can decrease the dimension of the NTRU lattice by about $2d$ in a similar fashion. This $2d$ reduction in lattice dimension should reduce LLL's running time by a factor exponential in $2d$.
To avoid the presented attack, $n$ should be chosen to be prime, or to have only large nontrivial factors \cite{Message Attack}.

We now consider this attack on BQTRU in the same context.
Indeed, for the bivariate case, the ring dimension  is $n^2$ which is composite ($n^2=n\cdot n$) and there are many suitable homomorphisms from the bivariate NTRU ring $\mathbb{Z}[x, y]/(x^n-1,  y^n-1)$ to a univariate NTRU ring $\mathbb{Z}[t]/(t^n -1)$;  suitable maps are all the maps $\phi_{r, s}: x\rightarrow t^r,  y\rightarrow t^s$ \cite{Orsini}. Therefore, the Gentry attack can be applied against BQTRU.
For considering this attack, we should map the  BQTRU lattice, from bivariate to univariate.
However, the problem is more difficult than in the NTRU case since the Euclidean part of the BQTRU lattice can pass to the quotient but this is not the case  for the Lagrange part $\mathfrak{D}$.
When we use a homomorphism to map $\lambda_{a, b}(x, y)$ to the univariate ring, then its image is not an univariate interpolator. In fact it is a product of two univariate interpolators \cite{Orsini}. Hence, while the the Euclidean part can be reduced from dimension $n^2$ to dimension $n$, the Lagrange part $\mathfrak{D}$ is reduced from $n^2$ to approximately $n^2/2$ dimension. As a consequence, solving the reduced lattice remains too hard and we  conjecture  that  BQTRU is practically secure against this attack. This problem is left for future study.

\section{Conclusions}~\label{concolusion}
In this paper, we present a NTRU-like  public key cryptosystem based on the quaternion algebras and bivariate polynomials.
We show that the multiplication operation in our cryptosystem can be done by means of matrix multiplication. Using Strassen's efficient method for $(2\times2)$ matrix multiplication, we obtain  $16/7$ times  faster operations in the  key  generation  and encryption  process of  our  cryptosystem compared to  NTRU  within comparable  parameters.

For designing a cryptosystem which is resistant against known lattice based attacks for low dimensions, we use bivariate polynomials ring and an ideal of this ring. This ideal is a lattice and the decryption process entails solving a CVP in this private lattice. Moreover, we show that the variance of the additive noise in this closest vector problem is small for an optimal decoder, like sphere decoder. Therefore, in the decryption process of our cryptosystem with the presented parameter sets,  a legitimate receiver using the private generator matrix, solves CVP  with affordable complexity and negligible error probability.

Due to the hybrid structure of the BQTRU lattice, recovering the private  key in our cryptosystem leads to solving SVP in the expanded lattice which is at least  $(2+\log_2 n)$ times larger than  the dimension of the NTRU lattice with equivalent parameters. This increasing in the dimension of the lattice force a high cost  to the attacker using a lattice attack in this cryptosystem. As a result of this expansion, this cryptosystem achieves high security against lattice attacks. Also we can reduce the size of the private keys without violating the system security and use larger message. Hence, message protection is feasible through larger polynomials increasing message security.
%

%

\end{document}